\documentclass{amsart}
\usepackage{amsmath,amssymb,latexsym}
\usepackage{amscd}

\newtheorem{theorem}{Theorem}[section]
\newtheorem{lemma}[theorem]{Lemma}
\newtheorem{proposition}[theorem]{Proposition}
\newtheorem{corollary}[theorem]{Corollary}

 \theoremstyle{definition}
 \newtheorem{definition}[theorem]{Definition}

 \DeclareMathOperator{\Ext}{Ext}
 \DeclareMathOperator{\Hom}{Hom}
 \DeclareMathOperator{\Tor}{Tor}
 \DeclareMathOperator{\colim}{colim}
 \DeclareMathOperator{\invlim}{lim}
 \DeclareMathOperator{\cok}{cok}
 \DeclareMathOperator{\im}{Im}
 \DeclareMathOperator{\Ho}{Ho}
 \DeclareMathOperator{\Stmod}{Stmod}

\newcommand{\uc}{\textup{:}}
\newcommand{\usc}{\textup{;}}
\newcommand{\ulp}{\textup{(}}
\newcommand{\urp}{\textup{)}}

\newcommand{\cat}[1]{\mathcal{#1}}           %% font for categories

\newcommand{\tensor}{\otimes}

\newcommand{\class}[1]{\mathcal{#1}}   %% font for classes
\newcommand{\N}{\mathbb{N}}
\newcommand{\Z}{\mathbb{Z}}
\newcommand{\Q}{\mathbb{Q/Z}}
\newcommand{\mathcolon}{\colon\,} %% Hovey uses for maps, like f: A -> B
\newcommand{\ar}{\xrightarrow{}} %% short hand for left arrows %%

\newcommand{\ch}{\textnormal{Ch}(R)}
\newcommand{\rmod}{R\text{-Mod}}
\newcommand{\ideal}[1]{\mathfrak{#1}}

\newcommand{\rightperp}[1]{#1^{\perp}}
\newcommand{\leftperp}[1]{{}^\perp #1}

\newcommand{\homcomplex}{\mathit{Hom}}

\begin{document}

\title{The stable module category of a general ring}

\author{Daniel Bravo}
 \address{Instituto de Ciencias F\'isica y Matem\'aticas \\
 Universidad Austral de Chile \\
 Campus Isla Teja \\
 Valdivia, Chile}
 \email{danielbravovivallo@gmail.com}

 \author{James Gillespie}
 \address{Ramapo College of New Jersey \\
          School of Theoretical and Applied Science \\
          505 Ramapo Valley Road \\
          Mahwah, NJ 07430}
 \email[Jim Gillespie]{jgillesp@ramapo.edu}
 \urladdr{http://phobos.ramapo.edu/~jgillesp/}

 \author{Mark Hovey}
 \address{Wesleyan University \\
          Middletown, CT 06459}
 \email{mhovey@wesleyan.edu}

%\extraline{acknowledge grant here}

\date{\today}

\begin{abstract}
For any ring $R$ we construct two triangulated categories, each
admitting a functor from $R$-modules that sends projective and
injective modules to $0$.  When $R$ is a quasi-Frobenius or Gorenstein
ring, these triangulated categories agree with each other and with the
usual stable module category.  Our stable module categories are
homotopy categories of Quillen model structures on the category of
$R$-modules.  These model categories involve generalizations of
Gorenstein projective and injective modules that we derive by
replacing finitely presented modules by modules of type $FP_{\infty}$.
\end{abstract}

\maketitle

\section{Introduction}

This paper is about the generalization of Gorenstein homological
algebra to arbitrary rings.  Gorenstein homological algebra first
arose in modular representation theory, where one looks at
representations of a finite group $G$ over a field $k$ whose
characteristic divides the order of $G$.  In this case, the group ring
$k[G]$ is quasi-Frobenius, which means that projective and injective
modules coincide.  It is natural, then, to ignore them and form the
stable module category $\Stmod (k[G])$ by identifying two
$k{G}$-module maps $f,g\mathcolon M\xrightarrow{}N$ when $g-f$ factors
through a projective module.  The stable module category is the main
object of study in modular representation theory.  For example, the
Tate cohomology of $G$ is just $\Stmod (k[G]) (k,k)_{*}$.

One could perform this construction for any ring $R$, but the
resulting category only has good properties when the ring is
quasi-Frobenius.   In this case, the resulting category $\Stmod (R)$
is not abelian, but it is triangulated.  The shift functor $\Sigma M$
is given by taking the cokernel of a monomorphism from $M$ into an
injective module; this is unique up to isomorphism in $\Stmod (R)$.
There is an exact coproduct-preserving functor
\[
\gamma \mathcolon \rmod \xrightarrow{} \Stmod (R),
\]
exact in the sense that if
\[
0 \xrightarrow{} M' \xrightarrow{f} M \xrightarrow{g} M'' \xrightarrow{} 0
\]
is a short exact sequence in $\rmod $, then there is an exact triangle
\[
M'\xrightarrow{} M \xrightarrow{g} M'' \xrightarrow{h} \Sigma M'
\]
in the triangulated category $\Stmod (R)$.  Note that $\gamma$ takes
projective (and injective, since they coincide) modules to $0$, and
also that $\gamma$ is universal with respect to this property.  That
is, given any other exact functor $\delta \mathcolon \rmod
\xrightarrow{}\cat{C}$ into a triangulated category that sends
projectives to $0$, there is a unique induced exact functor
$\overline{\delta}\mathcolon \Stmod (R)\xrightarrow{}\cat{C}$ such
that $\overline{\delta}\gamma =\delta$.  In fact, $\overline{\delta}$
is just $\delta$, since by assumption $\delta$ sends projective
modules to $0$ and therefore identifies $f$ and $g$ when their
difference factors through a projective module.

We have often wondered whether it would be possible to give a
construction of $\Stmod (R)$ for a general ring $R$, and it is the
goal of this paper to make such a construction.  We would like $\Stmod
(R)$ to have similar properties; it should be triangulated with an
exact functor $\gamma \mathcolon \rmod \xrightarrow{}\Stmod (R)$ that
sends projectives and injectives to $0$, and it should be somehow
universal with respect to this property.  It should also be natural
with respect to at least some ring homomorphisms, and of course it
should give the same answer in cases where the stable module category
has already been defined (this includes the case when $R$ is
Gorenstein~\cite{hovey-cotorsion}).

We began this paper by considering the stable derived category of
Krause~\cite{krause-stable}.  This
is the chain homotopy category of exact (unbounded) complexes of injective
modules.  Krause only considers this for Noetherian commutative rings $R$
(actually for Noetherian schemes), but most of his construction works
more generally.  The stable derived category is triangulated, and
there is an exact functor from $\rmod$ into it that sends projectives
to $0$, but it seems unlikely that it sends injectives to $0$.
(Explicit examples are hard to understand except in the easiest
cases).  There is also no hint that it is universal.  However, the
stable derived category does coincide with the stable module category
when $R$ is Gorenstein.

The stable derived category suggests approaching the stable module
category through unbounded complexes, rather than modules.  We give
two general constructions of Quillen model structures on unbounded
chain complexes, one where everything is cofibrant and the fibrant
objects are certain complexes of injectives, and one where everything
is fibrant and the cofibrant objects are certain complexes of
projectives.  We recover the stable derived category as the homotopy
category of one such model structure.  However, the key idea comes
from Gorenstein homological algebra, where one considers totally
acyclic complexes.  In the injective case, these are exact complexes
of injectives that remain exact after applying $\Hom (I,-)$ for any
injective module $I$.  This has always seemed a strange idea, but this
is the key to constructing a homotopy category where the injectives go
to $0$ as well as the projectives.

The definition of totally acyclic is really only appropriate for
Noetherian rings $R$; perhaps the main contribution of this paper is
describing how Gorenstein homological algebra should work for general
rings $R$.  The idea is as follows.  An injective module is a module
$I$ such that $\Ext^{1} (M,I)=0$ for all modules $M$, but in fact it
is equivalent to assume this for all finitely generated modules $M$.
But finitely generated modules are only well-behaved for Noetherian
rings.  One might try considering absolutely pure modules $I$ instead;
these are modules such that $\Ext^{1} (M,I)=0$ for all finitely
presented $M$.  But finitely presented modules are only well-behaved
over coherent rings.  So what we really need are modules which have
not only finite generators and relations, but finite relations between
the relations, and so on.  A module $M$ is said to have type
$FP_{\infty}$ if it has a projective resolution by finitely generated
projectives.  We can then define a module $I$ to be \textbf{absolutely
clean}, thinking of clean as a little weaker than
pure, if $\Ext^{1} (M,I)=0$ for all $M$ of type $FP_{\infty}$.
Finally, we define a complex $X$ of injectives to be \textbf{exact
AC-acyclic} if it is exact and $\Hom (I,X)$ remains exact for any
absolutely clean module $I$.  This is our generalization of
totally acyclic.

We then get a Quillen model structure on chain complexes where
everything is cofibrant and the fibrant objects are the exact
AC-acyclic complexes of injectives, for any ring $R$.  The homotopy
category of this model structure is our candidate for the stable
module category; it is the chain homotopy category of exact AC-acyclic
complexes of injectives.  Further evidence that this is the right thing is the
fact that it is indeed a homotopy category of modules, rather than
just complexes.  That is, there is a model structure on $\rmod $ whose
homotopy category is this stable module category.  In this model
structure, everything is cofibrant and the fibrant objects are the
\textbf{Gorenstein AC-injective} modules; these are the zero-cycles of
exact AC-acyclic complexes of injectives.  Our stable module category
is then equivalent to the quotient category of Gorenstein
AC-injective modules where two maps are identified when their
difference factors through an injective module.  Note, however, that
we need the model structure on complexes to approach the one on
modules; we cannot work directly with modules.

We have concentrated here on the injective case, but we also construct
analogous model structures using exact complexes of projectives.
There is a similar notion of an exact AC-acyclic complex of
projectives, and we get a likely different notion of a stable module
category. This is the chain homotopy category of exact AC-acyclic
complexes of projectives, or equivalently, the quotient category or
Gorenstein AC-projective modules obtained by identifying maps that
factor through a projective.

We have had to leave many loose ends in this paper that we hope to
address in future work.  We need more explicit examples.  It seems
likely to us that our two notions of stable module category likely
agree with some conditions on the ring; perhaps when the ring has a
dualizing module.  And our construction of the stable module category
includes a generalization of Tate cohomology to all algebras over $k$,
and in particular to all groups, finite or not.  Such generalizations
of Tate cohomology have been considered by many authors; in
particular, Benson uses modules of type $FP_{\infty}$ in his
generalization~\cite{benson-infinite}, but in a different way.  We would like to
understand the precise relationship between these approaches.  

This paper grew out of the Ph.\ D. thesis of the first
author~\cite{bravo}.  We also mention the work of Hanno
Becker~\cite{becker}, who independently constructed some of our model
categories in the Noetherian case and has a very nice interpretation
of the recollement constructed by Krause in terms of model categories.

Throughout this paper, $R$ will denote a ring with unity, and $\rmod$
will denote the category of left $R$-modules.  

\section{Finiteness, flatness, and injectivity}\label{sec-finite}

Homological algebra of any sort is about approximating objects by
projective, injective, and flat objects.  In this section, we point
out that the notions of flatness and injectivity depend on a choice of
which modules are considered to be finite, and that the usual choices
(finitely generated or finitely presented modules) may not be
appropriate for general rings.  Instead, we use the modules of type
$FP_{\infty}$ as our finite objects.  The ``injective'' modules for
this choice are the absolutely clean modules and the ``flat'' modules
are the level modules.  Over any ring $R$, these modules have
properties that injective modules only have over Noetherian rings and
flat modules only have over coherent rings.  They are also dual to
each other under taking character modules.  The drawback of this
approach is that for some non-coherent rings, the only modules of type
$FP_{\infty}$ are free, so that every module is absolutely clean and
level. 

More precisely, let us recall that a nonempty full subcategory
$\cat{D}$ of an abelian category is called \textbf{thick} if it is
closed under summands and whenever we have a short exact sequence
\[
0 \xrightarrow{} M'\xrightarrow{} M \xrightarrow{} M''\xrightarrow{}0
\]
where two of the three entries are in $\cat{D}$, then the third entry
is also in $\cat{D}$.  

The following proposition is well-known. 

\begin{proposition}\label{prop-thick}
Let $R$ be a ring. 
\begin{enumerate}
\item The category of finitely generated left $R$-modules is thick if
and only if $R$ is left Noetherian.  
\item The category of finitely presented left $R$-modules is thick if
and only if $R$ is left coherent.  
\end{enumerate}
\end{proposition}

\begin{proof}
Suppose the category of finitely generated left $R$-modules is thick
and $\ideal{a}$ is a left ideal.  Then the short exact sequence 
\[
0 \xrightarrow{} \ideal{a} \xrightarrow{} R \xrightarrow{} R/\ideal{a}
\xrightarrow{} 0
\]
shows that $\ideal{a}$ is finitely generated, so $R$ is left
Noetherian.  The converse is standard.  

Similarly, suppose the category of finitely presented left $R$-modules
is thick, and $\ideal{a}$ is a finitely generated ideal.  Then $R$ and
$R/\ideal{a}$ are finitely presented, so the same short exact sequence
shows that $\ideal{a}$ is finitely presented.  Thus $R$ is left
coherent.  The converse is also standard, but perhaps less well-known;
see~\cite[Theorem~2.5.1]{glaz}.  The hardest part is to prove that if
we have a short exact sequence
\[
0 \xrightarrow{} M' \xrightarrow{f} M \xrightarrow{g} M'' \xrightarrow{}0
\]
and $M$ and $M''$ are finitely presented, then $M'$ is as well.  It
suffices to prove that $M'$ is finitely
generated~\cite[Corollary~4.52]{lam}, and for this we take a
finite presentation 
\[
P''_{1} \xrightarrow{r''} P''_{0} \xrightarrow{q''} M'' \xrightarrow{}0
\] 
of $M''$ and a finitely generated projective module $P$ equipped with
a surjection $P\xrightarrow{q}M$.  There is then a map $\alpha
\mathcolon P\xrightarrow{}P''_{0}$ such that $q''\alpha =gq$.  Then 
\[
(\alpha ,r'')\mathcolon P\oplus P''_{1} \xrightarrow{} P''_{0}
\]
is a (necessarily split) surjection, and thus $\ker (\alpha ,r'')$ is
a finitely generated projective mapping onto $M$.  
\end{proof}

We interpret Proposition~\ref{prop-thick} as saying that finitely
presented is only the correct notion of finiteness when the
ring $R$ is left coherent.  In general, we need not just the
relations, but all the syzygies of a finite module to be finitely
generated.  That is, we need the following definition.  

\begin{definition}\label{defn-finite}
A (left) module $M$ over a ring $R$ is said to be of \textbf{type
$\mathbf{FP_{\infty}}$} if $M$ has a projective resolution by finitely
generated projectives.
\end{definition}

So if $R$ is left Noetherian, the modules of type $FP_{\infty}$ are
precisely the finitely generated modules, and if $R$ is left coherent,
the modules of type $FP_{\infty}$ are precisely the finitely presented
modules.

Modules of type $FP_{\infty}$ have been studied before, but mostly in
the context of geometric group theory.  The standard reference for
them is~\cite{bieri}.  They have been used in representation theory of
certain infinite groups by Benson in~\cite{benson-infinite} and
Kropholler~\cite{kropholler}, and their homological algebra has been
studied more generally in the Ph. D. thesis of Livia Miller (now Livia
Hummel)~\cite{miller-livia}.  

In particular, Bieri proves the following proposition. 

\begin{proposition}\label{prop-fp-thick}
For any ring $R$, the modules of type $FP_{\infty}$ form a thick
subcategory.  
\end{proposition}

So the modules of type $FP_{\infty}$ are a reasonable candidate for
the ``finite'' modules over any ring $R$.  However, there is no
guarantee that there are very many modules of type $FP_{\infty }$.  

\begin{proposition}\label{prop-fp-generate}
Suppose $R$ is a ring.  Every left $R$-module is a direct limit of
modules of type $FP_{\infty}$ if and only if $R$ is left coherent.  
\end{proposition}

\begin{proof}
Suppose that every module $M$ can be written as a direct limit
$\varinjlim M_{i}$ of modules of type
$FP_{\infty}$.  Then if $M$ is finitely presented we have 
\[
\Hom_{R} (M,M) \cong  \Hom_{R} (M, \varinjlim M_{i}) \cong \varinjlim
\Hom_{R} (M,M_{i}).  
\]
In particular, this means that the identity map of $M$ factors through
some $M_{i}$, so that $M$ is a summand of $M_{i}$ and is therefore
itself of type $FP_{\infty}$.  Thus the finitely presented modules
coincide with the modules of type $FP_{\infty}$ and so form a thick
subcategory, and so $R$ is left coherent.  The converse is immediate,
since every module is always a direct limit of finitely presented
modules.  
\end{proof}

To illustrate this proposition, we consider the following example. 

\begin{proposition}\label{prop-noncoherent-example}
Let $k$ be a field and let $R= k[x_{1},x_{2},\dotsc ]/\ideal{m}^{2}$
denote the quotient of the polynomial ring over $k$ on the generators
$x_{1},x_{2},\dotsc$ by the square of the maximal ideal $\ideal{m}=
(x_{1},x_{2},\dotsc)$.  Then the only $R$-modules of type
$FP_{\infty}$ are the finitely generated free modules.  
\end{proposition}

\begin{proof}
Suppose $M$ is finitely presented but not free, and write $M$ as the
quotient of two free modules
\[
R^{k} \xrightarrow{g} R^{n} \xrightarrow{} M \xrightarrow{} 0
\]
where $n$ is minimal and $k$ is minimal for that value of $n$ (and
positive since $M$ is not free).  Let us denote the standard
generators of $R^{k}$ by $e_{1},\dotsc ,e_{k}$ and the standard
generators of $R^{n}$ by $e_{1}',e_{2}',\dotsc ,e_{n}'$.  Because $n$
is minimal, we must have $\im g\subseteq \ideal{m}R^{n}$ for all $i$.
Indeed, if not, then if we write
\[
g (e_{i}) = \sum_{j} a_{j}e_{j}'
\]
then there must be an $i$ and a $j$ with $a_{j}$ not in $\ideal{m}$, so that
$a_{j}$ is a unit.  We can then solve for $e_{j}$.  The effect of this
is to write $M$ as the quotient 
\[
R^{k-1} \xrightarrow{h} R^{n-1} \xrightarrow{}M \xrightarrow{} 0
\]
where $h$ is obtained by eliminating $e_{i}$ from $R^{k}$ and $e_{j}'$
from $R^{n}$ and rewriting $g$.  

Now because $k$ is also minimal, we claim that $\ker g\subseteq
\ideal{m}R^{k}$.  Indeed, otherwise there is some $\sum a_{i}e_{i}$ in
the kernel of $g$ and some $j$ with $a_{j}$ not in $\ideal{m}$, so a
unit. We can then write this as 
\[
\sum a_{i}g (e_{i})=0
\]
and use this to solve for $g (e_{j})$ in terms of the other $gf
(e_{i})$.  This means that we can delete $e_{j}$ and rewrite $M$ as
the quotient 
\[
R^{k-1} \xrightarrow{h} R^{n} \xrightarrow{} M \xrightarrow{} 0
\]
where $h$ is the restriction of $g$, violating the minimality of
$k$.  

Now, since $\ideal{m}^{2}=0$ and $\im g\subseteq \ideal{m}R^{n}$, we
also have $\ideal{m}R^{k}\subseteq \ker g$.  Therefore, 
\[
\ker g = \ideal{m}R^{k}.  
\]
But $\ideal{m}R^{k}$ is not finitely generated, from which we conclude
that $M$ is not even of type $FP_{2}$, let alone of type
$FP_{\infty}$.  
\end{proof}

Now that we have a general notion of finiteness, there are
corresponding notions of flatness and injectivity.  More precisely, we
have the following definitions. 

\begin{definition}\label{defn-level}
Let $R$ be a ring.  A left $R$-module $N$ is called
\textbf{$FP_{\infty}$-injective} or \textbf{absolutely clean}
if $\Ext_{R}^{1} (M,N)=0$ for all modules $M$ of type $FP_{\infty}$.
Similarly, $N$ is called \textbf{level} if $\Tor^{1}_{R} (M,N)=0$ for
all right $R$-modules $M$ of type $FP_{\infty}$.  
\end{definition}

Absolutely clean modules are analogous to absolutely pure modules, and
coincide with them when $R$ is left coherent.  More precisely, a short
exact sequence $E$ is pure if and only if $K\otimes_{R}E$ is exact for
any right $R$-module $K$, and this is equivalent to $\Hom_{R} (M,E)$
being exact for any finitely presented left $R$-module $M$.  It
follows that a module $N$ is FP-injective ($\Ext^{1}_{R} (M,N)=0$ for
all finitely presented $M$) if and only if $N$ is absolutely pure
(every short exact sequence beginning with $N$ is pure).  In analogy
with this, we define a short exact sequence $E$ to be \textbf{clean}
if $\Hom_{R} (M,E)$ is exact for all $M$ of type $FP_{\infty}$.  Then
$N$ is absolutely clean if and only if every short exact sequence
beginning with $N$ is clean.  

Of course, for the ring of Proposition~\ref{prop-noncoherent-example},
the only modules of type $FP_{\infty}$ are free, so every module is
both absolutely clean and level. 

Recall that injective $R$-modules are closed under direct limits if
and only if $R$ is left Noetherian, and absolutely pure $R$-modules
are closed under direct limits if and only if $R$ is left
coherent~\cite{stenstrom-fp}.  Absolutely clean modules are always
closed under direct limits, and share many other properties of
absolutely pure modules.

\begin{proposition}\label{prop-finite-injective}
For any ring $R$, the following hold\uc
\begin{enumerate}
\item If $N$ is an absolutely clean $R$-module, then $\Ext^{n}
(M,N)=0$ if $n>0$ and $M$ is of type $FP_{\infty}$.
\item The class of absolutely clean modules is closed under pure
submodules and pure quotients.
\item The class of absolutely clean modules is coresolving;
that is, it contains the injective modules and is closed under
extensions and cokernels of monomorphisms.  
\item The class of absolutely clean modules is closed under products
and direct limits, and so also under transfinite
extensions. 
\item There is a set $S$ of absolutely clean modules such that every
absolutely clean module is a transfinite extension of modules in $S$.  
\end{enumerate}
\end{proposition}

For the last two statements, given a collection of modules $\cat{D}$, we
say that $N$ is a \textbf{transfinite extension} of objects in
$\cat{D}$ if there is an ordinal $\lambda$ and a colimit-preserving
functor $X\mathcolon \lambda \xrightarrow{}\rmod$ with $X_{0}\in
\cat{D}$ such that each map $X_{i}\xrightarrow{}X_{i+1}$ is a monomorphism
whose cokernel is in $\cat{D}$, and such that
$\colim_{i<\lambda}X_{i}\cong N$.  

\begin{proof}
Given a module $M$ of type $FP_{\infty }$, let $P_{*}$ be a resolution
of $M$ by finitely generated projectives.  Define $M_{0}=M$ and
$M_{i}=B_{i-1}P_{*}$ for $i>0$.  Then if we truncate $P_{*}$ by
setting everything below $i$ to $0$, we get a projective resolution of
$M_{i}$ by finitely generated projectives, so $M_{i}$ is also of type
$FP_{\infty}$. Since we have short exact sequences
\[
0 \xrightarrow{} M_{i} \xrightarrow{} P_{i-1} \xrightarrow{} M_{i-1}
\xrightarrow{} 0,
\]
we see that
\[
\Ext^{n} (M,N) = \Ext^{1} (M_{n-1},N),
\]
from which the first statement follows.

For the second statement, if the short exact sequence
\[
E: 0 \xrightarrow{} N' \xrightarrow{} N \xrightarrow{} N'' \xrightarrow{}0
\]
is pure and $N$ is absolutely clean, then $\Hom (M,E)$ is
still exact for any module $M$ of type $FP_{\infty }$.  Hence
$\Ext^{1} (M,N')$ is a submodule of the zero module $\Ext^{1} (M,N)$,
and so $N'$ is absolutely clean.  Then $\Ext^{1} (M,N'')$ is a
submodule of $\Ext^{2} (M,N')$, which is also zero by the first part
of the proposition.  Hence $N''$ is also absolutely clean.  

Now suppose
\[
0 \xrightarrow{} N' \xrightarrow{} N \xrightarrow{} N'' \xrightarrow{} 0
\]
is a short exact sequence where $N',N$ are absolutely clean.  By
applying $\Hom (M,-)$ to this short exact sequence, where $M$ is
$FP_{\infty}$, we see that $\Ext^{1} (M,N'')$ is trapped between the
two zero groups $\Ext^{1} (M,N)$ and $\Ext^{2} (M,N')$.  Hence it is
zero, and so $N''$ is absolutely clean.  A similar argument shows that
absolutely clean modules are closed under extensions, giving us the
third statement.

It is clear that absolutely clean modules are closed under products.
Now suppose $N_{\alpha}$ is a directed family of absolutely clean
modules.  Then
\begin{gather*}
\Ext^{1} (M,\varinjlim N_{\alpha}) \cong H^{1} (\Hom (P_{*},
\varinjlim N_{\alpha})) \cong H^{1} (\varinjlim \Hom
(P_{*},N_{\alpha})) \\
\cong \varinjlim H^{1}\Hom (P_{*},N_{\alpha})
\cong \varinjlim \Ext^{1} (M, N_{\alpha }),
\end{gather*}
giving us the fourth statement.  The last statement is an immediate
consequence of Proposition~\ref{prop-transfinite} below.  
\end{proof}

We owe the reader the following useful proposition.  

\begin{proposition}\label{prop-transfinite}
Suppose $\cat{A}$ is a class of $R$-modules that is closed under
taking pure submodules and quotients by pure submodules.  Then there
is a cardinal $\kappa$ such that every module in $\cat{A}$ is a
transfinite extension of modules in $\cat{A}$ with cardinality less
than $\kappa$.
\end{proposition}

The proof below does not require $\cat{A}$ to be closed under
transfinite extensions, but in practice one usually wants this as
well, so that a module is in $\cat{A}$ if and only if it is a
transfinite extension of modules in $\cat{A}$ with cardinality less
than $\kappa$.

\begin{proof}
Take $\kappa$ to be any cardinal larger than $|R|$.  Given $M\in
\cat{A}$, we define a strictly increasing chain $M_{i}\subseteq M$
with $M_{i}\in \cat{A}$ by transfinite induction on $i$.  For $i=0$,
we let $M_{0}$ be a nonzero pure submodule of $M$ of cardinality less
than $\kappa$ (using~\cite[Lemma~5.3.12]{enochs-jenda-book}).  Having
defined the pure submodule $M_{i}$ of $M$ and assuming that $M_{i}\neq
M$, we let $N_{i}$ be a nonzero pure submodule of $M/M_{i}$ of
cardinality less than $\kappa$. Since $M_{i}$ is a pure submodule of
$M$, $M/M_{i}$ is also in $\cat{A}$, so $N_{i}$ is as well.  We then
let $M_{i+1}$ be the preimage in $M$ of $N_{i}$, so that $M_{i+1}$ is
a pure submodule of $M$ so also in $\cat{A}$.  For the limit ordinal
step, we define $M_{\beta}= \bigcup_{\alpha <\beta}M_{\alpha}$; this
is a colimit of pure submodules of $M$ so is also a pure submodule of
$M$.  This process will eventually stop when $M_{i}=M$, at which point
we have written $M$ as a transfinite extension of modules in $\cat{A}$
with cardinality less than $\kappa$.
\end{proof}

We then have the following corollary to
Proposition~\ref{prop-finite-injective}.  

\begin{corollary}\label{cor-ac-coherent}
A ring $R$ is left coherent if and only if absolutely clean left
$R$-modules and absolutely pure left $R$-modules coincide.  
\end{corollary}

\begin{proof}
If absolutely clean and absolutely pure modules coincide, then
absolutely pure modules are closed under direct limits, so $R$ is left
coherent.  
\end{proof}

We will see below that level modules are dual to absolutely clean
modules, so should expect dual properties.  

\begin{proposition}\label{prop-level}
For any ring $R$, the following hold\uc 
\begin{enumerate}
\item If $N$ is a level $R$-module, then $\Tor_{n}^{R} (M,N)=0$ if
$n>0$ and $M$ is of type $FP_{\infty}$.  
\item The class of level modules is closed under pure submodules and pure
quotients.  
\item The class of level modules is resolving; that is, it contains
the projective modules and is closed under extensions and cokernels of
epimorphisms.
\item The class of level modules is closed under products and direct
limits, and so also under transfinite extensions.  
\item There is a set $S$ of level modules such that every
level module is a transfinite extension of modules in $S$.  
\end{enumerate}
\end{proposition}

\begin{proof}
For part~(a), suppose $M$ is of type $FP_{\infty}$ and take a
projective resolution $P_{*}$ of $M$ by finitely generated
projectives.  As we have seen before, this gives us short exact
sequences
\[
0 \xrightarrow{} M_{i+1} \xrightarrow{} P_{i} \xrightarrow{} M_{i}
\xrightarrow{} 0
\]
for all $i$, with $M_{0}=M$ and $M_{i+1}=Z_{i}P$.  Each $M_{i}$ is
then of type $FP_{\infty}$, because we can truncate $P_{*}$ to get a
resolution of $M_{i}$.  Then
\[
\Tor_{n}^{R} (M,F) = \Tor_{1}^{R} (M_{n-1},F)=0
\]
for $n>1$.

For the second statement, if the short exact sequence
\[
E: 0 \xrightarrow{} N' \xrightarrow{} N \xrightarrow{} N'' \xrightarrow{}0
\]
is pure and $N$ is level, then $M\otimes_{R}E$ is
exact for any right $R$-module $M$ of type $FP_{\infty }$.  Hence
$\Tor_{1}(M,N'')$ is a quotient of the zero module $\Tor_{1}(M,N)$,
and so $N''$ is level.  Then $\Tor_{1}(M,N')$ is a
quotient of $\Tor_{2}(M,N')$, which is also zero by the first part
of the proposition.  Hence $N'$ is also level.  

Now suppose
\[
0 \xrightarrow{} N' \xrightarrow{} N \xrightarrow{} N'' \xrightarrow{} 0
\]
is a short exact sequence where $N,N''$ are level.  By applying
$M\otimes_{R}-$ to this short exact sequence, where $M$ is
$FP_{\infty}$, we see that $\Tor_{1}(M,N')$ is trapped between the two
zero groups $\Tor_{1}(M,N)$ and $\Tor_{2} (M,N'')$.  Hence it is zero,
and so $N'$ is level.  A similar argument shows that level modules are
closed under extensions, giving us the third statement.

Because $\Tor$ commutes with colimits, it is clear that level modules
are closed under direct limits.  Now let $N_{i}$ be level for all $i$,
and let $M$ be a right module $M$ of type $FP_{\infty}$.  Take a
projective resolution $P_{*}$ of $M$ by finitely generated
projectives.  Then
\begin{gather*}
\Tor_{1}^{R} (M,\prod N_{i}) = H_{1} (P_{*}\otimes \prod N_{i}) \\
\cong H_{1} (\prod (P_{*}\otimes N_{i})) \cong \prod H_{1}
(P_{*}\otimes N_{i}) = 0,
\end{gather*}
where we have used the fact that each $P_{n}$ is finitely presented to
move it inside the product~\cite[Proposition~4.44]{lam}.  Finally, the
last statement is immediate from Proposition~\ref{prop-transfinite}.   
\end{proof}

\begin{corollary}\label{cor-level-coherent}
A ring $R$ is right coherent if and only if level \ulp left\urp $R$-modules
and flat $R$-modules coincide.
\end{corollary}

\begin{proof}
If $R$ is right coherent, then every finitely presented right $R$-module has type
$FP_{\infty}$, so a level module $N$ has $\Tor_{1}^{R} (M,N)=0$ for all
finitely presented $M$.  Since every module is a direct limit of
finitely presented modules, and $\Tor_{1}^{R} (-,N)$ preserves direct
limits, $N$ is flat. Conversely, if every level module is flat, then
Proposition~\ref{prop-level} shows that products of flat left
$R$-modules are flat.  This forces $R$ to be right coherent by Chase's
theorem~\cite[Theorem~4.47]{lam}.
\end{proof}

Now recall the well-known fact that a left $R$-module $N$ is flat if
and only if its character module $N^{+}=\Hom_{\Z} (M,\Q )$ is
injective (or, equivalently, absolutely pure) as a right
$R$-module~\cite[Theorem~4.9]{lam}, and the less well-known fact that
if $R$ is left Noetherian then $N$ is injective if and only if $N^{+}$
is flat as a right
$R$-module~\cite[Corollary~3.2.17]{enochs-jenda-book}.  This partial
duality between flat and injective modules is a reflection of a
perfect duality between level and absolutely clean modules.

\begin{theorem}\label{thm-duality}
For any ring $R$, a module $N$ is level if and only if $N^{+}$ is
absolutely clean, and $N$ is absolutely clean if and only if $N^{+}$
is level.  
\end{theorem}

\begin{proof}
Suppose $N$ is level and $M$ is a right $R$-module of type
$FP_{\infty}$.  Consider a short exact sequence
\[
E\mathcolon 0 \xrightarrow{} M_{1} \xrightarrow{} P \xrightarrow{} M
\xrightarrow{} 0
\]
where $P$ is projective.  Because $N$ is level, $E\otimes_{R}N$ is
exact.  Hence $(E\otimes_{R}N)^{+}$ is also exact, but by adjointness
this is the same as $\Hom_{R} (E,N^{+})$.  It then follows that
$\Ext^{1}_{R} (M,N^{+})=0$, so $N^{+}$ is absolutely clean.
Conversely, if $N^{+}$ is absolutely clean, then $\Hom_{R} (E,
N^{+})\cong (E\otimes_{R}N)^{+}$ is exact. Since $\Q $ is an
injective cogenerator of the category of abelian groups,
$E\otimes_{R}N$ must be exact, and so $\Tor^{R}_{1} (M,N)=0$ and $N$
is level.

Now suppose that $N$ is absolutely clean and $M$ is a left $R$-module
of type $FP_{\infty}$.  We take a short exact sequence 
\[
E\mathcolon 0 \xrightarrow{} M_{1} \xrightarrow{} P \xrightarrow{} M
\xrightarrow{} 0
\]
where $P$ is finitely generated projective and $M_{1}$ is also of
type $FP_{\infty}$, and in particular finitely presented. Since $N$ is
absolutely clean, $\Hom_{R} (E,N)$ is exact, and so $(\Hom_{R}
(E,N))^{+}$ is also exact.  Since $E$
consists of finitely presented modules and $\Q $ is an injective
%% \Q  should be \Q /\Z  I think
abelian group, we can apply Theorem~3.2.11
of~\cite{enochs-jenda-book} to conclude that
\[
(\Hom_{R} (E,N))^{+}\cong N^{+}\otimes_{R} E.
\]
It then follows immediately that $\Tor_{1}^{R} (N^{+},M)=0$, so
$N^{+}$ is level.  As before, the converse is a matter of reversing
the steps.  Indeed, if
$N^{+}$ is level, then $N^{+}\otimes_{R}E$ is exact, so $(\Hom_{R}
(E,N))^{+}$ is exact, so $\Hom_{R} (E,N)$ is exact.  We conclude that
$\Ext^{1}_{R} (M,N)=0$ and so $N$ is absolutely clean.
\end{proof}

Note that this theorem implies that if $R$
is left coherent and $N$ is an injective left $R$-module, then $N^{+}$
is level, so must be flat.  On the other hand, if $R$ is left coherent
but not left Noetherian, then there is an absolutely pure
module $N$ that is not injective, and it too will have $N^{+}$ flat.  

Now we recall the notion of a complete cotorsion pair.  Given an
abelian category $\cat{A}$, a \textbf{cotorsion pair} is a pair of
classes of objects $(\class{F},\class{C})$ of $\cat{A}$ such that
$\rightperp{\class{F}} = \class{C}$ and $\class{F} =
\leftperp{\class{C}}$. Here $\rightperp{\class{F}}$ is the class of
objects $Y \in \cat{A}$ such that $\Ext^1(F,Y) = 0$ for all $F \in
\class{F}$, and similarly $\leftperp{\class{C}}$ is the class of
objects $X \in \cat{A}$ such that $\Ext^1(X,C) = 0$ for all $C \in
\class{C}$. Two simple examples of cotorsion pairs in $\rmod$ are
$(\class{P},\class{A})$ and $(\class{A},\class{I})$, where $\class{P}$
is the class of projectives, $\class{I}$ is the class of injectives
and $\class{A}$ is the class of all $R$-modules. The canonical example
of a cotorsion pair is $(\class{F},\class{C})$ where $\class{F}$ is
the class of flat $R$-modules and $\class{C}$ is the class of
cotorsion $R$-modules~\cite{enochs-jenda-book}.

The cotorsion pair is said to have \textbf{enough projectives} if for
any $A \in \cat{A}$ there is a short exact sequence $0 \xrightarrow{}
C \xrightarrow{} F \xrightarrow{} A \xrightarrow{} 0$ where $C \in
\class{C}$ and $F \in \class{F}$. We say it has \textbf{enough
injectives} if it satisfies the dual statement. These two statements
are in fact equivalent for a cotorsion pair as long as the
\emph{category} $\cat{A}$ has enough projectives and
injectives~\cite[Proposition 7.1.7]{enochs-jenda-book}. We say that
the cotorsion pair is \textbf{complete} if it has enough projectives and
injectives. The book~\cite{enochs-jenda-book} is a standard reference
for cotorsion pairs.

Since level modules are analogous to flat modules, we should expect
them to be the left half of a complete cotorsion pair. 

\begin{definition}\label{defn-cotorsion}
A module $N$ is called \textbf{cospiral} if $\Ext^{1} (F,N)=0$ for all
level modules $F$.
\end{definition}

\begin{theorem}\label{thm-cospiral}
For any ring $R$, the pair (level modules, cospiral modules) forms
a complete cotorsion pair that is cogenerated by a set.  In
particular, the level modules form a covering class.
\end{theorem}

\begin{proof}
By Proposition~\ref{prop-level}, there is a set $S$ of level modules
such that the class of level modules is precisely the transfinite
extensions of $S$.  Then $S$ cogenerates a complete cotorsion theory
$(\cat{D},\cat{E})$, where $\cat{D}$ is the class of all summands of
transfinite extensions of elements of $S$; that is, the level modules.
Then $\cat{E}$ is necessarily the class of cospiral modules.
Completeness of the cotorsion theory proves that the class of level
modules is precovering.  Since it is also closed under directed
colimits, it is covering~\cite[Corollary~5.2.7]{enochs-jenda-book}.
\end{proof}

\section{Chain complexes}

We recall some basics about chain complexes and model categories in
this section.  Recall that $R$ is a ring with unity, and $\rmod$ is
the category of left $R$-modules.  We will denote the category of
unbounded chain complexes of (left) $R$-modules by $\ch$.  A chain
complex $\cdots \xrightarrow{} X_{n+1} \xrightarrow{d_{n+1}} X_{n}
\xrightarrow{d_n} X_{n-1} \xrightarrow{} \cdots$ will be denoted by
$(X,d)$ or simply $X$. We say $X$ is \textbf{bounded below (above)} if
$X_{n} = 0$ for $n < k$ ($n > k)$ for some $k \in \Z$.  We say it is
\textbf{bounded} if it is bounded above and below.  The \textbf{nth
cycle module} is defined as $\ker{d_{n}}$ and is denoted $Z_{n}X$. The
\textbf{nth boundary module} is $\im{d_{n+1}}$ and is denoted
$B_{n}X$. The \textbf{nth homology module} is defined to be
$Z_{n}X/B_{n}X$ and is denoted $H_{n}X$. Given an $R$-module $M$, we
let $S^{n}(M)$ denote the chain complex with all entries 0 except $M$
in degree $n$. We let $D^{n}(M)$ denote the chain complex $X$ with
$X_{n} = X_{n-1} = M$ and all other entries 0.  All maps are 0 except
$d_{n} = 1_{M}$. Given $X$, the \textbf{suspension of $X$}, denoted
$\Sigma X$, is the complex given by $(\Sigma X)_{n} = X_{n-1}$ and
$(d_{\Sigma X})_{n} = -d_{n}$.  The complex $\Sigma (\Sigma X)$ is
denoted $\Sigma^{2} X$ and inductively we define $\Sigma^{n} X$ for
all $n \in \Z$.

Given two chain complexes $X$ and $Y$ we define $\homcomplex(X,Y)$ to
be the complex of abelian groups $ \cdots \xrightarrow{} \prod_{k \in
\Z} \Hom(X_{k},Y_{k+n}) \xrightarrow{\delta_{n}} \prod_{k \in \Z}
\Hom(X_{k},Y_{k+n-1}) \xrightarrow{} \cdots$, where $(\delta_{n}f)_{k}
= d_{k+n}f_{k} - (-1)^n f_{k-1}d_{k}$.  This gives a functor
$\homcomplex(X,-) \mathcolon \ch \xrightarrow{} \textnormal{Ch}(\Z)$
which is left exact, and exact if $X_{n}$ is projective for all $n$.
Similarly the contravariant functor $\homcomplex(-,Y)$ sends right
exact sequences to left exact sequences and is exact if $Y_{n}$ is
injective for all $n$.

Recall that $\Ext^1_{\ch}(X,Y)$ is the group of (equivalence classes)
of short exact sequences $0 \ar Y \ar Z \ar X \ar 0$ under the Baer
sum. We let $\Ext^1_{dw}(X,Y)$ be the subgroup of $\Ext^1_{\ch}(X,Y)$
consisting of those short exact sequences which are split in each
degree. The next lemma is very useful. It is standard and we will not
prove it.

\begin{lemma}\label{lemma-homcomplex-basic-lemma}
For chain complexes $X$ and $Y$, we have
$$\Ext^1_{dw}(X,\Sigma^{(-n-1)}Y) \cong H_n \homcomplex(X,Y) =
Ch(R)(X,\Sigma^{-n} Y)/\sim \ ,
$$ where $\sim$ is chain homotopy.
\end{lemma}

In particular, for chain complexes $X$ and $Y$, $\homcomplex(X,Y)$ is
exact iff for any $n \in \Z$, any $f \mathcolon \Sigma^nX \ar Y$ is
homotopic to 0 (or iff any $f \mathcolon X \ar \Sigma^nY$ is homotopic
to 0).

Next recall that a model category is a category $\cat{M}$ with all
small limits and colimits equipped with three classes of maps called
cofibrations, fibrations, and weak equivalences, all subject to a list
of axioms allowing one to formally introduce homotopy theory on
$\cat{M}$. We assume the reader has a basic understanding or interest
in model categories. Standard references
include~\cite{hovey-model}
and~\cite{dwyer-spalinski}. In~\cite{hovey-cotorsion}, the third author found
the following 1-1 correspondence between complete cotorsion pairs
(discussed above just after Theorem~\ref{thm-duality}) and
\textbf{abelian} model category structures on $\cat{A}$. 

\begin{theorem}\label{them-Hoveys correspondence}
An abelian model structure on a bicomplete abelian category $\cat{A}$
is equivalent to a thick subcategory $\class{W}$ and two classes
$\class{Q}$ and $\class{R}$ for which $(\class{Q}, \class{R} \cap
\class{W})$ and $(\class{Q} \cap \class{W}, \class{R})$ are complete
cotorsion pairs. In this case $\class{W}$ is the class of trivial
objects, $\class{Q}$ the cofibrant objects and $\class{R}$ the fibrant
objects. Here an abelian model structure is one which a map is a
(trivial) cofibration if and only if it is a monomorphism with
(trivially) cofibrant cokernel. Equivalently, a map is a (trivial)
fibration if and only if it is a surjection with (trivially) fibrant
kernel.
\end{theorem}

The following two Propositions are really just Corollaries of
Theorem~\ref{them-Hoveys correspondence}. We will use them to
construct many different model structures on $\ch$.

\begin{proposition}[(Construction of an injective model
structure)]\label{prop-how to create an injective model structure}
Let $\cat{A}$ be a bicomplete abelian category with enough injectives
and denote the class of injectives by $\class{I}$. Let $\class{F}$ be
any class of objects and set $\class{W} =
\leftperp{\class{F}}$. Suppose the following conditions hold:
\begin{enumerate}
\item $(\class{W},\class{F})$ is a complete cotorsion pair.
\item $\class{W}$ is thick.
\item $\class{I} \subseteq \class{W}$.
\end{enumerate}
Then there is an abelian model structure on $\cat{A}$ where every
object is cofibrant, $\class{F}$ is the class of fibrant objects, $\class{W}$
is the class of trivial objects, and $\class{I} = \class{F} \cap \class{W}$
is the class of trivially fibrant objects.
\end{proposition}

\begin{proof}
Let $\class{A}$ also denote the class of all objects in the
category. Then $(\class{A},\class{I})$ is a complete cotorsion pair
since the category has enough injectives. By Theorem~\ref{them-Hoveys
correspondence}, we only still need (i) $\class{A} \cap \class{W} =
\class{W}$, and (ii) $\class{F} \cap \class{W} = \class{I}$. But of
course (i) is true so we just need to prove (ii).

First we see $\class{F} \cap \class{W} \supseteq \class{I}$, because
$(\class{W},\class{F})$ being a cotorsion pair automatically implies
$\class{I} \subseteq \class{F}$, and also by assumption we have
$\class{I} \subseteq \class{W}$. Next we show $\class{F} \cap
\class{W} \subseteq \class{I}$. So suppose $X \in \class{F} \cap
\class{W}$. Then find a short exact sequence $0 \ar X \ar I \ar I/X
\ar 0$ where $I$ is injective. By hypothesis $I \in \class{W}$. But
$\class{W}$ is assumed to be thick, which means $X/I \in
\class{W}$. But now since $(\class{W},\class{F})$ is a cotorsion pair
the short exact sequence splits. Therefore $X$ is a direct summand of
$I$, proving $X \in \class{I}$.
\end{proof}

We also list the dual for easy reference.

\begin{proposition}[(Construction of a projective model
structure)]\label{prop-how to create a projective model structure}
Let $\cat{A}$ be a bicomplete abelian category with enough projectives
and denote the class of projectives by $\class{P}$. Let $\class{C}$ be
any class of objects and set $\class{W} =
\rightperp{\class{C}}$. Suppose the following conditions hold:
\begin{enumerate}
\item $(\class{C},\class{W})$ is a complete cotorsion pair.
\item $\class{W}$ is thick.
\item $\class{P} \subseteq \class{W}$.
\end{enumerate}
Then there is an abelian model structure on $\cat{A}$ where every
object is fibrant, $\class{C}$ are the cofibrant objects, $\class{W}$
are the trivial objects, and $\class{P} = \class{C} \cap \class{W}$
are the trivially cofibrant objects.
\end{proposition}

\section{Injective model structures on $\ch$}\label{sec-inj}

In this section, we give a general construction of some abelian model
structures on $\ch$ for which every object is cofibrant and the
fibrant objects are contained in complexes of injectives.  We then use
this to build many different model structures on $\ch$, and in
particular, to build one whose homotopy category is our injective
stable module category of $R$.

\begin{theorem}\label{thm-how to create injective on chain}
Given a ring $R$, let $A$ be a fixed left $R$-module.  Let $\class{F}$
be the class of $A$-acyclic complexes of injectives; that is, chain
complexes $F$ that are degreewise injective and such that $\Hom_{R}
(A,F) = \homcomplex (S^{0} (A), F)$ is exact.  Then there is a cofibrantly
generated abelian model structure on $\ch$ where every object is
cofibrant, $\class{F}$ is the class of fibrant objects,
$\class{W}=\leftperp{\class{F}}$ is the class of trivial objects, and
the injective complexes $\class{I}=\class{F}\cap \class{W}$ are the
trivially fibrant objects.  Furthermore, $\class{W}$ contains all
contractible complexes.  We call this model structure the
\textbf{$A$-acyclic injective model structure}.  The homotopy category
of the $A$-acyclic injective model structure is equivalent to the
chain homotopy category of $A$-acyclic complexes of injectives.
\end{theorem}

\begin{proof}
We use Proposition~\ref{prop-how to create an injective model
structure}.  So we must show $(\class{W},\class{F})$ is a complete
cotorsion pair, that $\class{W}$ is thick, and that
$\class{I}\subseteq \class{W}$.  Let
\[
S = \{D^{n} (R/\ideal{a})|n\in \Z, \ideal{a} \text{ a left ideal of} R
\} \cup \{S^{n} (A)|n\in \Z \}.
\]
Then
\[
\class{F} =\rightperp{S}.
\]
Indeed, one can readily check that
\[
\Ext^{1}_{\ch}(D^{n} (R/\ideal{a}),F)\cong \Ext^{1}_{R} (R/\ideal{a},F_{n}).
\]
Baer's criterion for injectivity then implies that $\rightperp{S}$
consists of the complexes of injectives $F$ such that
\[
\Ext^{1}_{\ch} (S^{n}A,F) =0
\]
for all $n$.  But because $F$ is a complex of injectives,
\[
\Ext^{1}_{\ch} (S^{n}A,F) = \Ext^{1}_{dw} (S^{n}A,F) =
H_{n-1}\homcomplex (S^{0} (A),F),
\]
and so $\class{F}=\rightperp{S}$ as claimed.  Anytime we have a set
$S$ in a Grothendieck category with enough projectives, then
$(\leftperp{(\rightperp{S})},\rightperp{S})$ is always a complete
cotorsion pair by~\cite[Theorem~2.4]{hovey-cotorsion}, so
$(\class{W},\class{F})$ is so.

To see that $\cat{W}$ is thick, first note that, because $\cat{F}$
consists of complexes of injectives,
Lemma~\ref{lemma-homcomplex-basic-lemma} implies that $X\in \cat{W}$
if and only if $\homcomplex (X,F)$ is acyclic for all $F\in \cat{F}$.  Now
suppose we have a short exact sequence
\[
0 \xrightarrow{} X \xrightarrow{} Y \xrightarrow{} Z \xrightarrow{} 0,
\]
where two out of three of the entries are in $\cat{W}$.  Now suppose
$F\in \cat{F}$.  Since $F$ is a complex of injectives, the resulting
sequence
\[
0 \xrightarrow{} \homcomplex (Z,F) \xrightarrow{} \homcomplex (Y,F) \xrightarrow{}
\homcomplex (X,F)\xrightarrow{} 0
\]
is still short exact.  Since two out of three of these complexes are
acyclic, so is the third.

Note that if $X$ is contractible, then $\homcomplex (X,F)$ is obviously
acyclic for any $F$, so $X\in \cat{W}$.  This model structure is
cofibrantly generated by the results of~\cite[Section~6]{hovey-cotorsion}.
Explicitly, the generating trivial cofibrations can be taken to be the
set of all $D^{n} (\ideal{a})\xrightarrow{}D^{n} (R)$ for $\ideal{a}$
a left ideal of $R$ and $n\in \Z$ together with the set of all $S^{n}
(K)\xrightarrow{}S^{n} (P)$ for $n\in \Z$, where
\[
0 \xrightarrow{} K \xrightarrow{} P \xrightarrow{} A \xrightarrow{} 0
\]
is exact and $P$ is projective.  The generating cofibrations can be
taken to be the $S^{n} (\ideal{a})\xrightarrow{}S^{n} (R)$ for
$\ideal{a}$ a left ideal of $R$ and $n\in \Z$.  We leave to the reader
the statement about the homotopy category.
\end{proof}

Note that the proof of this theorem shows that the complete cotorsion
pair is cogenerated by the $D^{n} (R/\ideal{a})$ and the $S^{n}
(A)$.  This means that $\cat{W}$ consists of all summands of
transfinite extensions of chain complexes of the form $D^{n}
(R/\ideal{a})$ and $S^{n}A$.  

Of course, $\cat{W}$ is also a thick subcategory.  So, for simplicity,
if an object $B$ is in the smallest thick subcategory that is closed
under transfinite extensions and contains some set $S$ of objects, we
will say that \textbf{$B$ is built from $S$}.

In particular, if a module $B$ is built from the given module $A$,
then the complex $S^{0} (B)$ is built from $S^{0} (A)$, so $S^{0}B$ is
in $\cat{W}$.  So if $F$ is fibrant in the $A$-acyclic injective model
category, then it is also fibrant in the $B$-acyclic model structure.
Therefore, the identity functor from the $A$-acyclic injective model
category to the $B$-acyclic injective model category preserves
fibrations, and the trivial fibrations are the same in the two model
structures.  We therefore get the following proposition.

\begin{lemma}\label{lem-dependence-injective}
Suppose $A$ and $B$ are left $R$-modules and $B$ is built from $A$.
Then the identity functor is a left Quillen functor from the
$B$-acyclic injective model structure to the $A$-acyclic injective
model structure; in fact the $A$-acyclic injective model structure is
a left Bousfield localization of the $B$-acyclic injective model
structure.  In particular, if $A$ is also built from $B$, then the
$A$-acyclic injective model structure coincides with the $B$-acyclic
injective model structure.
\end{lemma}

Before considering examples, we point out some basic properties of the
map from $\rmod$ to the homotopy category of the $A$-acyclic injective
model structure.

\begin{proposition}\label{prop-homotopy-cat}
Let $R$ be a ring and $A$ a left $R$-module.  Consider the composite
functor
\[
\gamma \mathcolon \rmod \xrightarrow{S^{0}} \ch \xrightarrow{} \Ho \ch
\]
from $R$-modules to the homotopy category of the $A$-acyclic injective
model structure.  Then $\gamma$ is an exact coproduct-preserving
functor to the triangulated category $\Ho \ch$.  The kernel of
$\gamma$ contains all modules built from $A$.
\end{proposition}

\begin{proof}
The homotopy category $\Ho \ch$ is triangulated because the shift is
an equivalence of categories and is also equivalent to the suspension
functor that exists in any model category (see Section~7.1
of~\cite{hovey-model} for a general discussion of when the
homotopy category of a pointed model category is triangulated).

Any monomorphism in $\ch$ is a cofibration in the $A$-acyclic
injective model structure, and so short exact
sequences in $\ch$ give rise to exact triangles in $\Ho \ch$. The
functor $S^{0}$ is exact, so we conclude that $\gamma$ is exact.

The kernel of $\gamma$ consists of all modules $M$ such that $S^{0}M$
is trivial in the $A$-acyclic injective model structure.
In view of Lemma~\ref{lemma-homcomplex-basic-lemma}, this is all
modules $M$ such that $\Hom (M,X)$ is exact for all $A$-acyclic
complexes of injectives $X$.  In view of the proof of
Lemma~\ref{lem-dependence-injective}, this includes all modules built
from $A$.
\end{proof}

The simplest case of the $A$-acyclic injective model structure is of
course when $A=0$.

\begin{corollary}\label{cor-complexes-of injectives}
For any ring $R$ there is a cofibrantly generated abelian model
structure on $\ch$, the \textbf{Inj model structure}, in
which every object is cofibrant and the fibrant objects are the
complexes of injectives.  The trivially fibrant objects coincide with
the injective complexes, and the homotopy category is the chain
homotopy category of complexes of injectives.
\end{corollary}

At the other extreme, we could take $A$ to be the direct sum of all the
finitely generated modules.  Since every module is a transfinite
extension of finitely generated modules, the fibrant objects in this
case would be exact complexes of injectives $X$ such that $\Hom (M,X)$ is
exact for all left $R$-modules $M$.  In particular, by taking
$M=Z_{n}X$, we see that the inclusion $Z_{n}X\xrightarrow{}X_{n}$
factors through $d_{n+1}\mathcolon X_{n+1}\xrightarrow{}X_{n}$.  This
means that $X_{n+1}\cong Z_{n+1}X\oplus Z_{n}X$, from which it follows
easily that $X$ is an injective complex.  So this is the model
structure in which every map is a weak equivalence, the cofibrations
are the monomorphisms, and the fibrations are the split epimorphisms
with injective kernel.

It is more interesting to take $A=R$, when we get the following
corollary. Note that any projective module is built from $R$, since
it is a summand of a direct sum of copies of $R$.

\begin{corollary}\label{cor-exact-Inj}
For any ring $R$ there is a cofibrantly generated abelian model
structure on $\ch$, the \textbf{exact injective model structure}, in
which every object is cofibrant and the fibrant objects are the exact
complexes of injectives.  The trivially fibrant objects coincide with
the injective complexes.  For this model structure, all projective
modules are sent to $0$ by the functor $\gamma$ of
Proposition~\ref{prop-homotopy-cat}.  The homotopy category is the
chain homotopy category of exact complexes of injectives, the stable
derived category of~\cite{krause-stable}, and is denoted $S(R)$.
\end{corollary}

Note that, in general, if we replace $A$ by $A\oplus R$, we change the
fibrant objects from the $A$-acyclic complexes of injectives to the
exact $A$-acyclic complexes of injectives.  We therefore sometimes
refer to the $A\oplus R$-acyclic injective model structure as the
\textbf{exact $A$-acyclic injective model structure}.

If $R$ is left Noetherian, we can take $A$ to be the direct sum of all
the indecomposable injectives.  This will give us a model structure
where the fibrant objects are complexes of injectives $X$ such that
$\Hom_{R} (I,X)$ is exact for all injective modules $I$.  Such
complexes are called \textbf{Inj-acyclic}, and if they are also exact
they are called \textbf{totally acyclic complexes of injectives}.
We then get the following corollary.

\begin{corollary}\label{cor-totally-acyclic-model}
For any left Noetherian ring $R$, there is a cofibrantly generated
abelian model structure on $\ch$, the \textbf{Inj-acyclic injective
model structure}, in which every object is cofibrant and the fibrant
objects are the Inj-acyclic complexes of injectives. There is also a
cofibrantly generated abelian model structure on $\ch$, the
\textbf{totally acyclic injective model structure}, in which every
object is cofibrant and the fibrant objects are the totally acyclic
complexes of injectives. In both of these model structures, the
trivially fibrant objects coincide with the injective complexes.  In
the Inj-acyclic injective model structure, all injective modules are
sent to $0$ by the functor $\gamma$ of
Proposition~\ref{prop-homotopy-cat}\usc in the totally acyclic
injective model structure, both projective modules and injective
modules are sent to $0$ by $\gamma$.  In both cases, the homotopy
category is the chain homotopy category of the fibrant objects.
\end{corollary}

We now consider the simplest case, when $R$ is Gorenstein.  Here we
are using Gorenstein in the non-commutative sense, so that $R$ is left
and right Noetherian and has finite self-injective dimension on the
left and the right.  The main property of Gorenstein rings that we
need is that the modules of finite projective dimension coincide with
the modules of finite injective dimension;
see~\cite[Chapter~9]{enochs-jenda-book}.

\begin{proposition}\label{prop-Gor}
Suppose $R$ is Gorenstein.  Then the exact injective model structure,
the totally acyclic injective model structure and the Inj-acyclic
injective model structure all coincide.
\end{proposition}

\begin{proof}
Consider a general complex of injectives $X$, and let $\class{E}$
denote the collection of all $R$-modules $M$ such that $\Hom_{R}
(M,X)$ is acyclic.  Note that $\class{E}$ is obviously closed under
direct sums and retracts.  We claim that $\class{E}$ is thick.
Indeed, if we have a short exact sequence
\[
0 \xrightarrow{} M' \xrightarrow{} M \xrightarrow{} M'' \xrightarrow{} 0
\]
there is an induced short exact sequence of complexes
\[
0 \xrightarrow{} \Hom (M'',X)\xrightarrow{}\Hom (M,X)
\xrightarrow{}\Hom (M',X)\xrightarrow{}0,
\]
because $X$ is a complex of injectives.  The long exact sequence in
homology now proves that $\class{E}$ is thick.

It follows that, if $X$ is an exact complex of injectives, then $\Hom
(M,X)$ is exact for all $M$ of finite projective dimension.  In
particular, if $R$ is Gorenstein, then every injective module has
finite projective dimension, so $X$ is totally acyclic.  Thus the
exact injective and totally acyclic injective model structures
coincide.

Similarly, if $X$ is Inj-acyclic, it follows that $\Hom (M,X)$
is exact for all $M$ of finite injective dimension.  If $R$ is
Gorenstein, then $R$ itself has finite injective dimension, and so $X$
is acyclic.  Thus the Inj acyclic and totally acyclic Inj model
structure coincide as well.
\end{proof}

We now give an example to show that the exact injective model
structure may differ from the totally acyclic injective model
structure if the ring $R$ is not Gorenstein.  Let $R=k[x,y]/
(x^{2},xy,y^{2})$ where $k$ is a field, so that $R$ is an Artinian
local ring of Krull dimension zero with nilpotent maximal ideal
$\ideal{m}= (x,y)$.  There is only one indecomposable injective $J$,
the injective hull of $R/\ideal{m}\cong k$.  One can easily check that
$J=R\oplus R/K$, where $K$ is generated by $(x,0),(y,-x)$, and
$(0,y)$.

\begin{proposition}\label{prop-example}
For $R=k[x,y]/ (x^{2},xy,y^{2})$ as above, every module is built from
$R$ and $J$.  Therefore every totally acyclic complex of injectives is
actually an injective complex, and so the homotopy category of the
totally acyclic injective model structure is trivial.  On the other
hand, there is an exact complex of injectives that is not totally
acyclic, so the homotopy category of the exact injective model
structure is non-trivial.  Similarly, there is a complex of injectives
that is Inj-acyclic but not totally acyclic.  
\end{proposition}

\begin{proof}
The injective envelope of $R$ is $J\oplus J$.  Indeed, we can write
$J=k\langle \alpha ,\beta ,\gamma \rangle$, where $x\alpha =y\beta =0$
and $y\alpha =\gamma =x\beta$.  The map $R\xrightarrow{}J\oplus J$ that takes
$1$ to $(\alpha ,\beta )$ is then a monomorphism.  The cokernel of
this map is $k\oplus k\oplus k$.  Therefore, $k$ is built from $R$ and
$J$.  Since $k$ is the only simple $R$-module and $R$ is Artinian,
every finitely generated module is built from $k$.  But then every
module is built from $k$.  It follows that if $X$ is a totally acyclic
complex of injectives, then $\Hom (M,X)$ is acyclic for any $M$. We
have seen in the discussion following Corollary~\ref{cor-complexes-of
injectives} that this means that $X$ is injective as a complex.

To construct an example of an exact complex of injectives that is not
an injective complex, so not totally acyclic, let
$X_{n}=\oplus_{i=1}^{\infty} J $ for each $n$, and define $d\mathcolon
X_{n}\xrightarrow{}X_{n-1}$ by $d_{\alpha_{i}}=\gamma_{2i-1}$ (that
is, send the $\alpha$ in $J_{i}$ to the $\gamma$ in $J_{2i-1}$) and $d
(\beta_{i})=\gamma_{2i}$.  One can easily check then that
\[
\ker d = \im d = \bigoplus_{i=1}^{\infty} k
\]
generated by the $\gamma_{i}$.  So the complex is exact and cannot be
an injective complex, since its cycles are not injective.

To find a complex of injectives $Y$ that is Inj-acyclic but not exact,
we define $Y_{n}=X_{n}$ but this time we define the differential by 
\[
d (\alpha_{2i-1})=\gamma_{i}, d (\beta_{2i-1})=0 = d (\alpha_{2i}), d
(\beta_{2i})=\gamma_{i}.
\]
This complex is obviously not exact, since $\alpha_{1}-\beta_{2}$ is a
cycle that is not a boundary.  On the other hand, $\Hom (J,Y)$ is a
countable sum of copies of $R$ in each dimension, with
differential $d (1_{2i-1})=y_{i}$ and $d (1_{2i})=x_{i}$.  One can
then check easily that $\Hom (J,Y)$ is exact, and therefore $Y$ is
Inj-acyclic.  
\end{proof}

In view of Section~\ref{sec-finite}, we can extend the Inj-acyclic and
totally acyclic injective model structures to general rings by taking
$A$ to be the direct sum of a set of absolutely clean modules as in
Proposition~\ref{prop-finite-injective}(e) that generate all others by
transfinite extensions.  This will give a model structure where the
fibrant objects are \textbf{AC-acyclic}, in the sense that $\Hom
(I,X)$ is exact for all absolutely clean modules $I$ (so in particular
for all injective modules $I$).

\begin{theorem}\label{thm-FP-injective}
For any ring $R$, there is a cofibrantly generated abelian model
structure on $\ch$, the \textbf{AC-acyclic model structure}, in which
every object is cofibrant and the fibrant objects are the AC-acyclic
complexes of injectives. There is also a cofibrantly generated abelian
model structure on $\ch$, the \textbf{exact AC-acyclic model
structure}, in which every object is cofibrant and the fibrant objects
are the exact AC-acyclic complexes of injectives. The homotopy
category of the \ulp exact\urp AC-acyclic model structure is the chain
homotopy category of \ulp exact\urp AC-acyclic complexes of
injectives.  For these model structures, all absolutely clean modules
are sent to $0$ by the functor $\gamma$ of
Proposition~\ref{prop-homotopy-cat}, and all projective modules are
sent to $0$ for the exact AC-acyclic model structure.  When $R$ is
left Noetherian, the AC-acyclic model structure coincides with the
Inj-acyclic injective model structure, and the exact AC-acyclic model
structure coincides with the totally acyclic injective model
structure.
\end{theorem}

Note that for the ring $R$ of
Proposition~\ref{prop-noncoherent-example}, every module is absolutely
clean, so the fibrant objects in the AC-acyclic model structure are
the injective complexes, and every map is a weak equivalence.  

We would like to acknowledge the work of Pinzon~\cite{pinzon}, whose
theorems about absolutely pure modules over left coherent rings led us
to the more general notion of absolutely clean modules, which
of course agree with absolutely pure modules over left coherent
rings.  Note that Pinzon's theorems generalize; for example, the class
of absolutely clean modules is covering for any ring $R$.

When $R$ is left Noetherian, Krause constructs a recollement involving
$K (Inj)$, the stable derived category $S (R)$, and the derived category $D (R)$
in~\cite{krause-stable}.  Half of
this recollement arises because $S (R)$ is a Bousfield localization of
$K (Inj)$.  The other half can also be interpreted in terms of model
structures; this has been done very nicely in~\cite{becker}. Inspired
by Becker's work, the second author has given a general study of when
these recollements happen in~\cite{gillespie-recollement}.

\section{The Gorenstein AC-injective model structure on modules}\label{sec-Gorenstein-inj}

Recall that the usual stable module category can be obtained directly
from a model category structure on modules.  This is explained for
quasi-Frobenius rings in~\cite[Setion~2.2]{hovey-model} and
for Gorenstein rings in~\cite{hovey-cotorsion}.  In this section, we prove that
this can be done for the exact AC-acyclic model structure of
the previous section.

Following the method used for Gorenstein rings in~\cite{hovey-cotorsion}, we
define an $R$-module $M$ to be \textbf{Gorenstein AC-injective} if
$M=Z_{0}X$ for some exact AC-acyclic complex of injectives.  If $R$ is
left Noetherian, the Gorenstein AC-injectives are the usual Gorenstein
injectives; if $R$ is left coherent, the Gorenstein AC-injectives are
the Ding injectives discussed in~\cite{gillespie-ding}.  We would
like to put an abelian model structure on $\rmod$ in which everything
is cofibrant and the fibrant objects coincide with the class
$\class{F}$ of Gorenstein AC-injectives.  This forces us to define
\[
\class{W} = \leftperp{\class{F}}.
\]

We need to understand how $\class{W}$ relates to the class of trivial
objects in the exact $A$-injective model structure.

\begin{lemma}\label{lem-cycles-of-W}
Let $A$ be an $R$-module, and suppose $X$ is a complex with $H_{i}X=0$
for $i<0$ and such that $X_{i}$ is absolutely clean for $i>0$.
Then $X$ is trivial in the exact AC-acyclic model structure if and
only if $Z_{0}X\in \class{W}$.
\end{lemma}

\begin{proof}
Suppose $M$ is Gorenstein AC-injective, so that $M=Z_{0}Y$ for some
exact AC-acyclic complex of injectives $Y$.  We claim that there is an
isomorphism
\[
\Ext^{1} (X,Y) \xrightarrow{} \Ext^{1} (Z_{0}X,M);
\]
this isomorphism would prove the lemma.  Indeed, because $Y$ is a
complex of injectives, Lemma~\ref{lemma-homcomplex-basic-lemma} gives
us an isomorphism
\[
\Ext^{1} (X,Y) \xrightarrow{} \ch (X, \Sigma Y)/\sim,
\]
where $\sim$ denotes chain homotopy.  A chain map $\phi \mathcolon
X\xrightarrow{}\Sigma Y$ induces a map $Z_{0}X\xrightarrow{}Z_{-1}Y$.
A chain homotopy between $\phi $ and $0$ gives us maps $D_{n}\mathcolon
X_{n}\xrightarrow{}Y_{n}$ with $-dD_{n}+D_{n-1}d=\phi _{n}$.  In
particular $\phi _{0}=-dD_{0}$ on $Z_{0}X$.  Thus, there is a natural map
\[
\tau \mathcolon \ch (X, \Sigma Y)/\sim \xrightarrow{} \Hom (Z_{0}X, Z_{-1}Y)/\Hom
(Z_{0}X,Y_{0})\cong \Ext^{1} (Z_{0}X, M).
\]

To see that $\tau $ is surjective, suppose we have a map $f\mathcolon
Z_{0}X\xrightarrow{}Z_{-1}Y$.  Because $Y_{-1}$ is injective, there is
a map $f_{0}\mathcolon X_{0}\xrightarrow{}Y_{-1}$ extending $f$.  We
therefore get an induced map $Z_{-1}X=B_{-1}X\xrightarrow{}B_{2}Y$
using the fact that $H_{-1}X=0$.  Because $Y_{-2}$ is injective, this
extends to a map $f_{-1}\mathcolon X_{-1}\xrightarrow{}Y_{-2}$.
Continuing in this fashion we can define maps $f_{n}$ for all $n\leq
0$.  To define $f_{n}$ in positive degrees, we use the fact that $Y$
is exact AC-acyclic.  Indeed, our given map $f$ induces the composite
\[
X_{1} \xrightarrow{d} Z_{0}X \xrightarrow{f} Z_{-1}Y.
\]
Since $X_{1}$ is absolutely clean and $Y$ is AC-acyclic, there is a map
$f_{1}\mathcolon X_{1}\xrightarrow{}Y_{0}$ such that $df_{1}$ is this
composite.  This means that $f_{1}d\mathcolon
X_{2}\xrightarrow{}Y_{0}$ lands in $Z_{0}Y$, and so, because $X_{2}$
is absolutely clean and $Y$ is AC-acyclic, there is a map
$f_{2}\mathcolon X_{2}\xrightarrow{}Y_{1}$.  Continuing in this
fashion, we complete $f$ to a chain map as required.

To see that $\tau $ is injective, suppose we have a chain map $\phi
\mathcolon X\xrightarrow{}\Sigma Y$ such that $Z_{0}\phi$ factors
through $Y_{0}$ via a map $g_{0}\mathcolon Z_{0}X\xrightarrow{}Y_{0}$.
We need to construct a chain homotopy $D$ between $\phi$ and $0$, and
we do this following the same plan.  As before, we extend $Z_{0}\phi$
to a map $D_{0}\mathcolon X_{0}\xrightarrow{}Y_{0}$.  Then
$dD_{0}-\phi_{0}$ induces a map $g_{-1}\mathcolon
Z_{-1}X=B_{-1}X\xrightarrow{}Z_{-1}Y$, which we extend to a map
$D_{-1}\mathcolon X_{-1}\xrightarrow{}Y_{-1}$, so that
$dD_{0}-D_{-1}d=\phi_{0}$.  We continue in this fashion to construct
$D_{n}$ for all negative $n$.  We then define $D_{n}$ for positive $n$
using the fact that $X_{i}$ is absolutely clean for $i>0$ and $Y$ is
AC-acyclic, as we did above.
\end{proof}

As a scholium of the above proof, we note the following proposition.

\begin{proposition}\label{prop-Gor-inj-induced}
For any ring $R$, suppose $M$ and $N$ are Gorenstein AC-injective
$R$-modules, with $M=Z_{0}X$ and $N=Z_{0}Y$ for $X$ and $Y$ exact
AC-acyclic complexes of injectives.  Given a map $f\mathcolon
M\xrightarrow{}N$, there is a chain map $\phi \mathcolon
X\xrightarrow{}Y$ with $Z_{0}\phi =f$.
\end{proposition}

We now use this proposition to prove the following lemma.

\begin{lemma}\label{lem-Gor-inj-retracts}
For any ring $R$, the collection of Gorenstein AC-injective
$R$-modules is closed under retracts.
\end{lemma}

The proof we give here is inspired by the proof for the corresponding
fact for Gorenstein injectives over a Noetherian ring given
in~\cite{holm}.

\begin{proof}
For any ring $R$, the exact AC-acyclic complexes of injectives are
closed under products, and so Gorenstein AC-injectives are as well.  The
Eilenberg swindle then implies that if $N$ is a retract of a
Gorenstein AC-injective, there is a Gorenstein AC-injective module $M$ with
$N\oplus M\cong M$.  Indeed, if $N\oplus N'$ is Gorenstein AC-injective,
we take $M=\prod _{i=1}^{\infty} (N\oplus N')$, so that
\[
M \cong N \oplus \prod_{i=1}^{\infty} (N'\oplus N) \cong
N\oplus M.
\]

Now choose an exact AC-acyclic complex of injectives $X$ with
$M=Z_{0}X$. By Proposition~\ref{prop-Gor-inj-induced}, the monomorphism
\[
f\mathcolon M\xrightarrow{} M\oplus N\cong M
\]
is induced by a chain map $\phi \mathcolon X\xrightarrow{}X$, so that
$Z_{0}\phi =f$.  We would like to take the cokernel of $\phi$ and
claim that it is an exact AC-acyclic complex of injectives with zero
cycles equal to $N$.  Unfortunately, $\phi$ need not be a
monomorphism.  So instead we use
\[
\psi \mathcolon X \xrightarrow{} X'=X\oplus \prod _{i\neq 0}
D^{i+1}X_{i},
\]
where $\psi$ has components $\phi$ and the maps $\rho_{i}\mathcolon
X\xrightarrow{}D^{i+1}X_{i}$ that are $d$ in degree $i+1$ and the
identity in degree $i$.  Note that $\psi$ is clearly a monomorphism in
every degree except possibly $0$, but also in degree $0$ because
$Z_{0}\phi$ is a monomorphism.  Furthermore, $Z_{0}\psi =f$.

We then get a short exact sequence
\[
0 \xrightarrow{} X \xrightarrow{\psi} X' \xrightarrow{} Y
\xrightarrow{} 0
\]
of complexes, necessarily degreewise split since $X$ is a complex of
injectives.  Both $X$ and $X'$ are exact AC-acyclic complexes of
injectives (since the disks we add are contractible), and so $Y$ is as
well.  Furthermore, because $X$ is exact, the induced sequence
\[
0 \xrightarrow{} Z_{0}X \xrightarrow{Z_{0}\psi } Z_{0}X'
\xrightarrow{} Z_{0}Y \xrightarrow{} 0
\]
is still exact, and so $Z_{0}Y\cong N$ and so $N$ is Gorenstein
AC-injective.
\end{proof}

\begin{lemma}\label{lem-spheres}
A module $M$ over a ring $R$ is in $\class{W}$ if and only if $S^{0}M$
is trivial in the exact AC-acyclic model structure.
\end{lemma}

\begin{proof}
A calculation using Lemma~\ref{lemma-homcomplex-basic-lemma} shows
that
\[
\Ext^{1} (S^{0}M, X) = \Ext^{1} (M, Z_{1}X).
\]
for any exact complex of injectives $X$.
\end{proof}

\begin{theorem}\label{thm-Gor-module}
For any ring $R$, there is an abelian model
structure on $\rmod$, the \textbf{Gorenstein AC-injective model structure},
in which every object is cofibrant and the fibrant objects are the
Gorenstein AC-injective modules.
\end{theorem}

This model structure generalizes the Gorenstein injective model
structure defined in~\cite{hovey-cotorsion} and also its generalization
in~\cite{gillespie-ding}.  Note, though, that we do not have an
explict cogenerating set for the cotorsion pair
$(\class{W},\class{F})$ when $R$ is not Gorenstein, though we will
prove below that a cogenerating set does exist.  When $R$ is
Gorenstein, the syzygies of the indecomposable injective modules
cogenerate.

One aspect of this model structure is that the class of Gorenstein
AC-injectives is special pre-enveloping for any ring $R$.

\begin{proof}
As above, we take $\class{F}$ to be the Gorenstein AC-injective
modules, and define $\class{W}=\leftperp{\class{F}}$.  Then
Lemma~\ref{lem-spheres} shows that $\class{W}$ is thick and contains
the injective modules.  Now for any module $M$, we have a short exact
sequence
\[
0 \xrightarrow{} S^{0}M \xrightarrow{}X \xrightarrow{}Y \xrightarrow{} 0
\]
in which $X$ is an exact AC-acyclic complex of injectives and $Y$ is
trivial in the exact AC-acyclic model structure.  Applying
$Z_{0}=\Hom_{\ch} (S^{0}R,-)$, we get an exact sequence
\[
0 \xrightarrow{} M \xrightarrow{} Z_{0}X \xrightarrow{} Z_{0}Y
\xrightarrow{} \Ext^{1}_{\ch} (S^{0}R, S^{0}M) =0.
\]
Of course $Z_{0}X$ is Gorenstein AC-injective by definition, but $Z_{0}Y$
is in $\class{W}$ as well by Lemma~\ref{lem-cycles-of-W}, since
$Y_{i}$ is injective for all $i\neq 0$ and $H_{i}Y=0$ for all $i\neq
1$.  So the purported cotorsion pair $(\class{W},\class{F})$ has
enough injectives, and hence enough projectives as well if it is a
cotorsion pair.

We can then use this to show that $\class{F}=\rightperp{\class{W}}$,
so that $(\class{W},\class{F})$ is in fact a cotorsion pair.
Indeed, suppose $M\in \rightperp{\class{W}}$.  Find a short exact
sequence
\[
0 \xrightarrow{} M \xrightarrow{} J \xrightarrow{} N \xrightarrow{} 0
\]
where $J$ is Gorenstein AC-injective and $N\in \class{W}$.  By
assumption, this must split, and so $M$ is a retract of $J$ and hence
is Gorenstein AC-injective by Lemma~\ref{lem-Gor-inj-retracts}.
Proposition~\ref{prop-how to create an injective model structure} now
completes the proof.
\end{proof}

It is enlightening to examine the homotopy category of the Gorenstein
AC-injective model structure, but to do so we need a lemma.

\begin{lemma}\label{lem-hereditary}
For any ring $R$, the cotorsion pair $(\cat{W},\cat{F})$, where $\cat{F}$ is the
Gorenstein AC-injective modules, is hereditary, so that $\Ext^{n}
(W,F)=0$ for all $n>0$, $W\in \cat{W}$, and $F\in \cat{F}$.  Hence the
collection of Gorenstein AC-injective modules is coresolving \ulp
closed under cokernels of monomorphisms\urp .
\end{lemma}

\begin{proof}
Suppose $F$ is Gorenstein AC-injective, so that $F=Z_{0}X$, where $X$
is an exact AC-acyclic complex of injectives.  We have short exact
sequences
\[
0 \xrightarrow{} Z_{i}X \xrightarrow{} X_{i} \xrightarrow{}
Z_{i-1}X\xrightarrow{}0
\]
for all $i$, and each $Z_{i}X$ is also Gorenstein AC-injective.  A
simple computation then shows that
\[
\Ext^{n} (M, F) = \Ext^{1} (M, Z_{-n+1}X),
\]
so $(\cat{W}, \cat{F})$ is hereditary.  It follows that $\cat{F}$ is
coresolving, just as in the proof of the third part of
Proposition~\ref{prop-finite-injective}.
\end{proof}

\begin{theorem}\label{thm-homotopy-Gorenstein}
For any ring $R$, the homotopy category of the Gorenstein AC-injective
model structure is the quotient category of the category of Gorenstein
AC-injective modules obtained by identifying two maps when their
difference factors through an injective module.
\end{theorem}

\begin{proof}
Since the Gorenstein AC-injectives are the cofibrant and fibrant
objects, the homotopy category is the quotient category of the
category of Gorenstein injectives by the homotopy relation.  Since
injective modules are trivial, any map that factors through an
injective module gets sent to $0$ in the homotopy category.  So if
$f-g$ factors through an injective, then $f=g$ in the homotopy
category, and so $f$ and $g$ are homotopic.  Conversely, suppose
$f,g\mathcolon M\xrightarrow{}N$ are homotopic, where $M$ and $N$ are
Gorenstein AC-injective.  Then there is a homotopy $H\mathcolon
M\xrightarrow{}N'$ between them for any path object $N'$ of $N$.  We
can obtain a path object by factoring the diagonal map
$N\xrightarrow{}N\times N$ into a trivial cofibration
$N\xrightarrow{i}N'$ followed by a fibration $N'\xrightarrow{p}N\times
N$.  In particular, since $p$ is a fibration and the Gorenstein
AC-injectives are closed under extensions (like the right half of any
cotorsion pair), we see that $N'$ is Gorenstein AC-injective.  The
cokernel of $i$ is therefore also Gorenstein AC-injective since the
Gorenstein AC-injectives are coresolving.  But this cokernel is in $\cat{W}$
as well, so the cokernel of $i$ is in fact an injective module.  If we
let $d\mathcolon N\times N\xrightarrow{}N$ denote the difference map,
we then have
\[
f-g = d\circ (f,g) = d\circ p\circ H.
\]
But $d\circ p\circ i=0$, so $d\circ p$ factors through the injective
module $\cok i$, and so $f-g$ does as well.
\end{proof}

The Gorenstein AC-injective model structure and the exact AC-acyclic
model structure are Quillen equivalent.

\begin{theorem}\label{thm-totally-Gorenstein}
For any ring $R$, the functor $S^{0}\mathcolon \rmod
\xrightarrow{}\ch$ is a Quillen equivalence from the Gorenstein
AC-injective model structure to the exact AC-acyclic model structure.
\end{theorem}

\begin{proof}
Since $S^{0}$ is exact, it preserves cofibrations.
Lemma~\ref{lem-spheres} shows that it preserves trivial cofibrations,
so it is a left Quillen functor with right adjoint $Z_{0}$.  We will
show that $Z_{0}$ reflects weak equivalences between fibrant objects
and that the map $M\xrightarrow{}Z_{0}RS^{0}M$ is a weak equivalence,
where $R$ denotes fibrant replacement in the exact AC-acyclic model
structure.  In view of Corollary~1.3.16
of~\cite{hovey-model}, this will complete the proof.

We first show that $Z_{0}$ reflects weak equivalences between fibrant
objects.  As with any right Quillen functor, it suffices to show that
if $f\mathcolon X\xrightarrow{}Y$ is a cofibration of fibrant objects
such that $Z_{0}f$ is a weak equivalence, then $f$ is a weak
equivalence.  So we are given a short exact sequence
\[
0 \xrightarrow{} X \xrightarrow{f} Y \xrightarrow{} C \xrightarrow{}0
\]
with $X$ and $Y$ exact AC-acyclic complexes of injectives.  This
sequence is necessarily degreewise split, from which it follows that
$C$ is also an exact AC-acyclic complex of injectives. The functor
$Z_{0}$ is left exact but not exact, but since $X$ is exact we do get
a short exact sequence
\[
0 \xrightarrow{} Z_{0}X \xrightarrow{} Z_{0}Y \xrightarrow{} Z_{0}C
\xrightarrow{} 0.
\]
Since $Z_{0}f$ is a weak equivalence, $Z_{0}C$ is in $\class{W}$.
Lemma~\ref{lem-cycles-of-W} then implies that $C$ is trivial in the
exact AC-acyclic model structure, and so $f$ is a weak equivalence.

Now take any module $M$, and let $RS^{0}M$ be a fibrant replacement
for $S^{0}M$, so that we have a short exact sequence
\[
0\xrightarrow{} S^{0}M \xrightarrow{} RS^{0}M \xrightarrow{} Y
\xrightarrow{} 0
\]
with $Y$ trivial in the exact AC-acyclic model category.  Applying
the functor $Z_{0}=\Hom_{\ch} (S^{0}R,-)$, we get an exact sequence
\[
0 \xrightarrow{} M \xrightarrow{} Z_{0}RS^{0}M \xrightarrow{} Z_{0}Y
\xrightarrow{} \Ext^{1}_{\ch} (S^{0}R, S^{0}M) = 0.
\]
Furthermore, $Y_{i}$ is injective for all $i\neq 0$ and $H_{i}Y=0$ for
all $i\neq 1$, so Lemma~\ref{lem-cycles-of-W} implies that $Z_{0}Y\in
\class{W}$.  Hence $M\xrightarrow{}Z_{0}RS^{0}M$ is a weak
equivalence.
\end{proof}

In view of Theorem~\ref{thm-totally-Gorenstein} , the functor
\[
\gamma : \rmod \xrightarrow{}\Ho \rmod
\]
to the homotopy category of the Gorenstein AC-injective model
structure is an exact functor to a triangulated category that
preserves coproducts and sends all absolutely clean and all
projective modules to $0$.

In fact, it is also initial in a sense.

\begin{proposition}\label{prop-initial}
The homotopy category of the Gorenstein AC-injective model structure
is initial among all triangulated categories with an exact functor
from $\rmod$ that preserves coproducts and sends all elements of
$\cat{W}$ to zero.
\end{proposition}

\begin{proof}
Suppose we have an exact functor $F\mathcolon \rmod
\xrightarrow{}\cat{C}$ that sends all elements of $\cat{W}$ to $0$.
Then $F$ sends all trivial cofibrations and fibrations to
isomorphisms, since these are each part of exact sequences where the
other term is in $\cat{W}$, so map to exact triangles where one term
is $0$.  Since every weak equivalence is a composition of a trivial
cofibration followed by a trivial fibration, $F$ sends all weak
equivalences to isomorphisms.  Hence $F$ extends uniquely to the
homotopy cateory.
\end{proof}

To truly understand the Gorenstein AC-injective model structure, then,
we need to know what $\cat{W}$ is.  All we know in general is that
$\cat{W}$ contains all absolutely clean and all projective
modules, and that $\cat{W}$ is a thick subcategory closed under direct
limits (as the kernel of any $F$ as in Proposition~\ref{prop-initial}
must be).  So we would like to know, for example, that $\cat{W}$ is
the smallest thick subcategory containing all absolutely clean
and projective modules and closed under direct limits.  But we do not
know this.

We can at least prove that $(\cat{W},\cat{F})$ is cogenerated by a
set.

\begin{proposition}\label{prop-cogenerated}
For any ring $R$, the cotorsion pair $(\cat{W}, \cat{F})$, where
$\cat{F}$ is the class of Gorenstein AC-injectives, is cogenerated by
a set.  Thus the Gorenstein AC-injective model structure is
cofibrantly generated.
\end{proposition}

We use the work of {\v{S}}{\v{t}}ov{\'{\i}}{\v{c}}ek~\cite{stovicek}
on deconstuctibility.  The following lemma is not stated
explicitly in~\cite{stovicek}, so we prove it here, but it is implicit
there.

\begin{lemma}\label{lem-stovicek}
If $(\cat{A},\cat{B})$ is a cotorsion pair in a Grothendieck category
that is cogenerated by a set, then there exists a set $S\subseteq
\cat{A}$ such that every element of $\cat{A}$ is a transfinite
extension of objects of $S$.
\end{lemma}

Recall that by definition there is a set $T$ of elements of $\cat{A}$
such that every element of $\cat{A}$ is a summand in a transfinite
extension of objects of $T$. This lemma removes the summand condition
at the expense of making $T$ larger.  We should also note that
although this lemma has a similar conclusion as
Proposition~\ref{prop-transfinite}, its hypotheses are quite
different.  

\begin{proof}
Let $T$ be a cogenerating set as above, and let $\cat{D}$ be the class
of all transfinite extensions of objects of $T$.  By definition, this
is a deconstructible class in the sense of~\cite{stovicek}.  Now
$\cat{A}$ is the class of all summands of objects of $\cat{D}$, but
{\v{S}}{\v{t}}ov{\'{\i}}{\v{c}}ek~\cite[Proposition~2.9]{stovicek}
proves that this means that $\cat{A}$ is also deconstructible.  Hence
there is a set $S$ such that $\cat{A}$ is the collection of all
transfinite extensions of elements of $\cat{A}$.
\end{proof}

\begin{proof}[Proof of Proposition~\ref{prop-cogenerated}]
Take a set $S'$ of complexes as in Lemma~\ref{lem-stovicek} for the
cotorsion pair $(\cat{W}',\cat{F}')$, where $\cat{F}'$ is the class of
exact AC-acyclic complexes of injectives. Now let $S$ be the
collection of all modules $M$ such that $S^{0}M\in S'$.  Then
$S\subseteq \cat{W}$ by Lemma~\ref{lem-spheres}, and if $N\in
\cat{W}$, then $S^{0}N\in \cat{W}'$ by the same lemma, so $N$ must be
a transfinite extension of objects of $S'$.  However, each term
$X_{\alpha}$ in this transfinite composition is a subobject of
$S^{0}N$, so must be $S^{0}M_{\alpha}$ for some module $M_{\alpha}$.
It follows that $M$ is a transfinite extension of objects in $S$.
\end{proof}

\section{Projective model structures on $\ch$}\label{sec-proj}

Having constructed model structures based on complexes of injective
modules, it is natural to try to make similar constructions with
complexes of projective modules. The projective case is harder to deal
with because we must still prove that our cotorsion pairs are
cogenerated by a set, and there is no natural choice for such a set
even in the simplest case, as there is no dual version of Baer's
criterion for injectivity.  We follow the standard idea of using all
objects in the left half of the cotorsion pair that are not too big as
a cogenerating set.

The goal of this section, then, is to state the following theorem,
analogous to Theorem~\ref{thm-how to create injective on chain}, and
derive some analogous corollaries.  The proof of Theorem~\ref{thm-how
to create projective on chain} is technical and will be postponed
until the next section.

\begin{theorem}\label{thm-how to create projective on chain}
Given a ring $R$, let $A$ be a given right $R$-module.  Let
$\class{C}$ be the class of $A$-acyclic complexes of projectives; that
is, chain complexes $C$ that are degreewise projective and such that
$A\otimes_{R}C$ is exact.  Then there is a cofibrantly generated
abelian model structure on $\ch$ where every object is fibrant,
$\class{C}$ is the class of cofibrant objects, $\class{W} =
\rightperp{\class{C}}$ is the class of trivial objects, and the
projective complexes $\class{P} = \class{C} \cap \class{W}$ are the
trivially cofibrant objects. Furthermore, $\class{W}$ contains all
contractible complexes.  We call this model structure the
\textbf{$A$-acyclic projective model structure}.  Its homotopy category
is equivalent to the chain hommotopy category of $A$-acyclic complexes
of projectives.
\end{theorem}

We have the dual dependence on $A$ that we had with the $A$-acyclic
injective model structure.

\begin{lemma}\label{lem-dependence-projective}
Suppose $A$ and $B$ are left $R$-modules and $B$ is built from $A$.
Then the identity functor is a left Quillen functor from the
$A$-acyclic projective model structure to the $B$-acyclic projective
model structure; in fact the $A$-acyclic projective model structure is
a right Bousfield localization of the $B$-acyclic projective model
structure.  In particular, if $A$ is also built from $B$, then the
$A$-acyclic projective model structure coincides with the $B$-acyclic
projective model structure.
\end{lemma}

\begin{proof}
Given a complex of projectives $C$, the collection of all right
modules $M$ such that $M\otimes_{R}C$ is exact is a thick subcategory
that is closed under transfinite extensions.  Hence if $C$ is
$A$-acyclic it is also $B$-acyclic.  It follows that the identity
functor from the $A$-acyclic projective model structure to the
$B$-acyclic one preserves cofibrations, and they have the same trivial
cofibrations.
\end{proof}

The homotopy category of the $A$-acyclic projective model structure
has similar properties to that of the $A$-acyclic injective model
structure.

\begin{proposition}\label{prop-homotopy-cat-proj}
Let $R$ be a ring and $A$ a right $R$-module.  Consider the composite
functor
\[
\gamma \mathcolon \rmod \xrightarrow{S^{0}} \ch \xrightarrow{} \Ho \ch
\]
from $R$-modules to the homotopy category of the $A$-acyclic projective
model structure.  Then $\gamma$ is an exact product-preserving
functor to the triangulated category $\Ho \ch$.  The kernel of
$\gamma$ consists of all modules $M$ such that $\Hom_R (X, M)$ is
exact for all $A$-acyclic complexes of projectives $X$.
\end{proposition}

\begin{proof}
The homotopy category $\Ho \ch$ is triangulated because the inverse
shift is an equivalence of categories and is also equivalent to the
loop functor that exists in any model category (see Section~7.1
of~\cite{hovey-model} for a general discussion of when the
homotopy category of a pointed model category is triangulated).

Any epimorphism in $\ch$ is a fibration in the $A$-acyclic
projective model structure, and so short exact
sequences in $\ch$ give rise to exact triangles in $\Ho \ch$. The
functor $S^{0}$ is exact, so we conclude that $\gamma$ is exact.

The kernel of $\gamma$ consists of all modules $M$ such that $S^{0}M$
is trivial in the $A$-acyclic projective model structure.
In view of Lemma~\ref{lemma-homcomplex-basic-lemma}, this is all
modules $M$ such that $\Hom_R (X,M)$ is exact for all $A$-acyclic
complexes of projectives $X$.
\end{proof}

The simplest case of the $A$-acyclic projective model structure is
when $A=0$.

\begin{corollary}\label{cor-complexes-of-projectives}
For any ring $R$, there is a cofibrantly generated abelian model
structure on $\ch$, the \textbf{Proj model structure}, in which the
cofibrant objects are the complexes of projectives and every object is
fibrant.  The trivially cofibrant objects coincide with the projective
complexes, and the homotopy category is the chain homotopy category
of complexes of projectives.
\end{corollary}

At the other extreme, we could take $A$ to be the direct sum of all
finitely generated right $R$-modules.  At first glance this appears to
be more interesting than in the injective case. Here the cofibrant
objects are the complexes of projectives $C$ such that $M\otimes_{R}C$
is exact for all right $R$-modules $M$. So these are the pure exact
complexes of projectives. However, as we show in
Theorem~\ref{thm-pure}, any pure exact complex is trivial in the Proj
model structure. So a pure exact complex of projectives is trivially
cofibrant in the Proj model structure, and hence is a projective
complex.  Thus, this is the model structure in which every map is a
weak equivalence, the fibrations are the epimorphisms, and the
cofibrations are the split monomorphisms with projective cokernel.

We can take $A=R$ to get the following corollary.

\begin{corollary}\label{cor-exact-Proj}
For any ring $R$ there is a cofibrantly generated abelian model
structure on $\ch$, the \textbf{exact projective model structure}, in which
the cofibrant objects are the exact complexes of projectives and every
object is fibrant.  The trivially cofibrant objects coincide with the
projective complexes and the homotopy category is the chain homotopy
category of exact complexes of projectives.  For this model structure,
all injective modules are sent to $0$ by the functor $\gamma$ of
Proposition~\ref{prop-homotopy-cat-proj}.
\end{corollary}

The last sentence follows from
Lemma~\ref{lemma-homcomplex-basic-lemma}, since $\Hom_R (C,I)$ is
always exact if $C$ is exact and $I$ is injective.

Just as in the injective case, replacing $A$ by $A\oplus R$ changes
the cofibrant objects from $A$-acyclic complexes of projectives to
exact $A$-acyclic complexes of projectives, so we sometimes call the
$A\oplus R$-acyclic projective model structure the \textbf{exact
$A$-acyclic projective model structure}.

Now, if $R$ is right Noetherian, we can take $A$ to be the direct sum
of all indecomposable injectives and $R$ itself.  One might expect
this to give us the totally acyclic projective model structure.  But
an exact complex of projectives $C$ is defined to be totally acyclic
if $\Hom (C,P)$ is exact for all projective modules $P$.  We get
instead exact complexes of projectives such that $\Hom (C,F)$ is exact
for all \emph{flat} modules $F$.  Furthermore, this result extends to
right coherent rings, and has a natural extension to all rings.  

\begin{definition}\label{defn-firmly}
Let $R$ be a ring and $C$ a complex of projective $R$-modules.  We say
that $C$ is \textbf{AC-acyclic} if $I\otimes_{R}C$ is exact for all
absolutely clean right $R$-modules $I$.  We say that $C$ is
\textbf{firmly acyclic} if $\Hom (C,F)$ is exact for all level
left $R$-modules $F$.  If $C$ is itself exact, we will call $C$
\textbf{exact AC-acyclic} or \textbf{exact firmly acyclic} as the
case may be.
\end{definition}

We will then prove the following theorem in the appendix, as
Corollary~\ref{cor-Inj-exact-level}.  

\begin{theorem}\label{thm-Inj-exact-level}
For any ring $R$, a complex of projectives $C$ is AC-acyclic if and
only if it is firmly acyclic.  If level $R$-modules all have
finite projective dimension, these conditions are equivalent to
$\Hom_{R} (C,P)$ being exact for all projective $R$-modules $P$.
\end{theorem}

If $R$ is right Noetherian, this theorem says that a complex of
projectives $C$ has $I\otimes_{R}C$ exact for all injective right
$R$-modules $I$ if and only if $\Hom (C,F)$ is exact for all flat left
$R$-modules $F$.  This was proved by Murfet and Salarian
in~\cite{murfet-salarian}.  If $R$ is right coherent,
Theorem~\ref{thm-Inj-exact-level} says that a complex of projectives
$C$ has $I\otimes_{R}C$ exact for all absolutely pure right
$R$-modules $I$ if and only if $\Hom_{R} (C,F)$ is exact for all flat
left $R$-modules $F$, extending the Murfet-Salarian theorem to the
coherent case.

For a general ring $R$, as in the injective case, we can take $A$ to
be the direct sum of all the absolutely clean modules of small
cardinality to get the following theorem.

\begin{theorem}\label{thm-AC-projective}
For any ring $R$, there is a cofibrantly generated abelian model
structure on $\ch$, the \textbf{firmly acyclic model structure}, in
which every object is fibrant and the cofibrant objects are the
firmly acyclic complexes of projectives.  There is also a
cofibrantly generated abelian model structure on $\ch$, the
\textbf{exact firmly acyclic model structure}, in which every
object is fibrant and the cofibrant objects are the exact firmly
acyclic complexes of projectives.  The homotopy category of the \ulp
exact\urp firmly acyclic projective model structure is the chain
homotopy category of \ulp exact\urp firmly acyclic complexes of
projectives.  The functor $\gamma$ of
Proposition~\ref{prop-homotopy-cat-proj} sends all injective modules
to $0$ in both the firmly acyclic and exact firmly acyclic
model structures, and also sends level modules to $0$ in the exact
firmly acyclic model structure.
\end{theorem}

To get a totally acyclic projective model structure, we need to know
slightly more about $R$.

\begin{corollary}\label{cor-totally-proj}
Suppose $R$ is a ring in which all level modules have finite
projective dimension.  Then a complex of projectives $C$ is cofibrant in
the \ulp exact\urp firmly acyclic model structure if and only if
$C$ is \ulp exact\urp totally acyclic.  If $R$ is right coherent, it
suffices for all flat modules to have finite projective dimension.
\end{corollary}

Note that there are many rings in which every flat module has finite
projective dimension.  This is true if $R$ is right Noetherian and has a
dualizing complex~\cite{jorgensen-flat}.

Just as in the injective case, most of these model structures coincide
when $R$ is Gorenstein.

\begin{proposition}\label{prop-Gor-proj}
Suppose $R$ is Gorenstein.  Then the exact projective model structure,
the firmly acyclic model structure, and the exact firmly
acyclic model structure all coincide.
\end{proposition}

\begin{proof}
The proof is very similar to the injective case.  So consider a
general complex of projectives $X$, and let $\class{E}$ denote the
collection of all right $R$-modules $M$ such that $M\otimes_{R}X$ is
acyclic.  As in the injective case, it is easy to check that
$\class{E}$ is thick.  It follows that, if $X$ is an exact complex of
projectives, then $M\otimes_{R}X$ is exact for all $M$ of finite
projective dimension.  In particular, if $R$ is Gorenstein, then every
injective module has finite projective dimension, so $X$ is
firmly acyclic.  Thus the exact projective and the exact
firmly acyclic model structures coincide.

Similarly, if $X$ is firmly acyclic, it follows that
$M\otimes_{R}X$ is exact for all $M$ of finite injective dimension.
If $R$ is Gorenstein, then $R$ itself has finite injective dimension,
and so $X$ is exact.  Thus the firmly acyclic and exact
firmly acyclic model structures coincide as well.
\end{proof}

The same example as we used in the injective case shows that these
model structures do not coincide in general if $R$ is not Gorenstein.

\begin{proposition}\label{prop-proj-example}
Suppose $R=k[x,y]/ (x^{2},xy,y^{2})$ where $k$ is a field.  Then every
totally acyclic complex of projectives is a projective complex, so the
homotopy category of the firmly acyclic projective model
structure is trivial.  On the other hand, there is an exact complex of
projectives that is not firmly acyclic, so the homotopy category
of the exact projective model structure is not trivial.  Similarly,
there is a firmly acyclic complex of projectives that is not exact.  
\end{proposition}

Note that this ring $R$ has a dualizing module $k$, so in particular a
dualizing complex, so a complex of projectives is exact firmly
acyclic if and only if it is totally acyclic.

\begin{proof}
If $X$ is a totally acyclic complex of projectives, then it is exact
and firmly acyclic, so $I\otimes_{R}X$ is exact for any injective
module $I$.  In particular, both $R\otimes_{R}X$ and $J\otimes_{R}X$
are exact, where $J$ is the unique indecomposable injective.  Since
every $R$-module is built from $R$ and $J$ by
Proposition~\ref{prop-example}, we see that $M\otimes_{R}X$ is exact
for all modules $M$, so that $X$ is a pure exact complex of
projectives. As pointed out after
Corollary~\ref{cor-complexes-of-projectives} this immediately implies
$X$ is projective. But in this case since the projective modules
coincide with the flat modules it is easy to see directly that each
cycle group $Z_nX$ is projective.

To find an exact complex of projectives that is not totally acyclic,
let $X_{n}=\bigoplus_{i=1}^{\infty}R$, and define $d\mathcolon
X_{n}\xrightarrow{}X_{n-1}$ by $d (1_{2i-1})=x_{i}$ and $d
(1_{2i})=y_{i}$.  Then one can easily that
\[
\ker d = \im d = \bigoplus_{i=1}^{\infty } k
\]
generated by the $x_{i}$ and the $y_{i}$, so $X$ is an exact complex
of projectives.  It is not a projective complex since the cycles are
not projective.

To find a firmly acyclic complex $Y$ of projectives that is not exact, we
let $Y_{n}=X_{n}$ with differential $d (1_{i})=x_{2i}+y_{2i-1}$.  Then
$x_{1}$ is a cycle that is not a boundary, so $Y$ is not exact.  But
$J\otimes_{R}Y$ is a countable direct sum of copies of $J$, with $d
(\alpha_{i})=\gamma_{2i-1}$ and $d (\beta_{i})=\gamma_{21}$.  So
$J\otimes_{R}Y$ is exact, and this $Y$ is firmly acyclic. 
\end{proof}

\section{Proof of the projective model structures}

The proof of Theorem~\ref{thm-how to create projective on chain} is
technical.  The basic plan is to find well-behaved small subcomplexes
inside a complex $C$ for which $A\otimes_{R}C$ is exact.  The key fact
that makes this work is the result of Kaplansky~\cite{kaplansky-projective} that
every projective module is a direct sum of countably generated
projective modules.  We first consider how to cope with complexes,
each of whose degrees is a direct sum.

\begin{lemma}[(Covering Lemma)]\label{lemma-covering lemma}
Let $\kappa$ be an infinite cardinal and suppose $X$ is a nonzero
complex in which each $X_n$ has a direct sum decomposition $X_n =
\oplus_{i \in I_n} M_{n,i}$ where $|M_{n,i}| < \kappa$ for all $i \in
I_n$. Then for any choice of subcollections $J_n \subseteq I_n$ (at
least one of which is nonempty), with $|J_n| < \kappa$, we can find a
nonzero subcomplex $S \subseteq X$ with each $S_n = \oplus_{i \in K_n}
M_{n,i}$ for some subcollections $K_n \subseteq I_n$ satisfying $J_n
\subseteq K_n$ and $|K_n| < \kappa$.
\end{lemma}

\begin{proof}
Suppose we are given such subcollections $J_n \subseteq I_n$. First,
for each $n$, we may build a subcomplex $X^n$ of $X$ as follows: In
degree $n$ the complex will consist of $\oplus_{i \in J_n}
M_{n,i}$. Then noting $d(\oplus_{i \in J_n} M_{n,i}) \subseteq
\oplus_{i \in I_{n-1}} M_{n-1,i}$ we define $L_{n-1} = \{ \, i \in
I_{n-1} \, | \, d(\oplus_{i \in J_n} M_{n,i}) \cap M_{n-1,i} \neq 0 \,
\}$. This essentially ``covers'' $d(\oplus_{i \in J_n} M_{n,i})$ with
summands in the sense that $d(\oplus_{i \in J_n} M_{n,i}) \subseteq
\oplus_{i \in L_{n-1}} M_{n-1,i}$ and yet $|L_{n-1}| < \kappa$ because
$|d(\oplus_{i \in J_n} M_{n,i})| < \kappa$. Now the subcomplex of $X$
we are constructing will consist of $\oplus_{i \in L_{n-1}} M_{n-1,i}$
in degree $n-1$. We continue down in the same way finding $L_{n-2}
\subseteq I_{n-2}$ with $|L_{n-2}| < \kappa$ and with $d(\oplus_{i \in
L_{n-1}} M_{n-1,i}) \subseteq \oplus_{i \in L_{n-2}} M_{n-2,i}$. In
this way we get a subcomplex of $X$: $$X^n = \cdots \xrightarrow{} 0
\xrightarrow{} \oplus_{i \in J_{n}} M_{n,i} \xrightarrow{} \oplus_{i
\in L_{n-1}} M_{n-1,i} \xrightarrow{} \oplus_{i \in L_{n-2}} M_{n-2,i}
\xrightarrow{} \cdots $$ Finally set $X \cup_{l \in \N} X^l$ and note
that this complex, obviously nonzero because at least one $I_n \neq
\phi$, will work. (The sets $K_n$ we claim to exist are the union of
all the $J_n$'s and all the various $L_i$ in sight. We still have
$|K_n| < \kappa$.)
\end{proof}

Now that we find small subcomplexes of complexes with degreewise
direct sum decompositions, we need to find small exact subcomplexes of
exact complexes with degreewise direct sum decompositions.

\begin{lemma}[(Exact Covering Lemma)]\label{lemma-exact covering lemma}
Let $\kappa$ be an infinite cardinal and suppose $Y$ is an exact
complex in which each $Y_n$ has a direct sum decomposition $Y_n =
\oplus_{i \in I_n} M_{n,i}$ where $|M_{n,i}| < \kappa$ for all $i \in
I_n$. Then for any choice of subcollections $K_n \subseteq I_n$, with
$|K_n| < \kappa$, we can find an exact subcomplex $T \subseteq Y$ with
each $T_n = \oplus_{i \in J_n} M_{n,i}$ for some subcollections $J_n
\subseteq I_n$ satisfying $K_n \subseteq J_n$ and $|J_n| < \kappa$.
\end{lemma}

\begin{proof}
We prove this in two steps.

\noindent (Step 1). We first show the following: If $X \subseteq Y$ is
any exact subcomplex with $|X| < \kappa$, then for any single one of
the given $K_n$, we can find an exact subcomplex $T \subseteq Y$
containing $X$ and so that for this given $n$, $T_n = \oplus_{i \in
L_n} M_{n,i}$ for some $L_n \subseteq I_n$ with $K_n \subseteq L_n$
and $|L_n| < \kappa$.

For this given $n$, first set $D_n = \{ \, i \in I_{n} \, | \, X_n
\cap M_{n,i} \neq 0 \, \}$. Since $|X_n| < \kappa$, we have $|D_n| <
\kappa$. Now define $L_n = D_n \cup K_n$ and set $T_n = \oplus_{i \in
L_n} M_{n,i}$. Of course $|L_n| < \kappa$ and $X_n \subseteq T_n$.

So all we need to do is extend $T_n$ into an exact subcomplex
containing $X$ and with cardinality less than $\kappa$. We build down
by setting $T_{n-1} = S_{n-1} + d(T_n)$ and $T_i = S_i$ for all $i <
n-1$. One can check that \[T_n \ar S_{n-1} + d(T_n) \ar S_{n-2} \ar
\cdots\] is exact. In particular, we have exactness in degree $n-1$
since $d(S_n) \subseteq d(T_n)$.

Next we build up from $T_n$. To start, take the kernel of $T_n \ar
T_{n-1}$ and find a $T'_{n+1} \subseteq Y_{n+1}$ such that $|T'_{n+1}|
< \kappa$ and $T'_{n+1}$ maps surjectively onto this kernel. Then take
$T_{n+1} = S_{n+1} + T'_{n+1}$. Now $T_{n+1}$ also maps surjectively
onto this kernel. We continue upward to build $T_{n+2}, T_{n+3},
\cdots$ in the same way and we are done.

\

\noindent We now finish the proof. From Step 1, taking $X = 0$ and the
subcollection to be $K_0$ we can find an exact subcomplex $T^0
\subseteq Y$ such that $(T^0)_0 = \oplus_{i \in L_0} M_{0,i}$ for some
$L_0 \subseteq I_0$ with $K_0 \subseteq L_0$ and $|L_0| < \kappa$.
Now using Step 1 again, with $X = T^0$ and using $K_{-1}$, we get
another exact subcomplex $T^1$ containing $T^0$ and such that
$(T^1)_{-1} = \oplus_{i \in L_{-1}} M_{{-1},i}$ for some $L_{-1}
\subseteq I_{-1}$ with $K_{-1} \subseteq L_{-1}$ and $|L_{-1}| <
\kappa$. Lets say that $T^0$ was constructed using a ``degree 0
operation'' and $T^1$ was constructed using a ``degree -1
operation''. Then we can continue to use ``degree $k$ operations''
with the following back and forth pattern on $k$:
\[0, \ \ -1,0,1, \ \ -2,-1,0,1,2, \ \ -3, -2, -1, 0, 1,2,3 \ \
\cdots\] to build an increasing union of exact subcomplexes, $\{\, T^l
\,\}$.  Finally set $T = \cup_{l \in \N} T^l$. Then by a cofinality
argument we see that for each $n$ we have $T_n = \oplus_{i \in J_n}
M_{n,i}$ for some subsets $J_n \subseteq I_n$ (the $J_n$'s are each a
countable union of the newly constructed $L_n$'s obtained in each
``pass'', and so $|J_n| < \kappa$). Clearly each $K_n \subseteq J_n$
and $T$ is exact.
\end{proof}

With these lemmas in hand, we turn to our specific situation.  Let $P$
be any complex of projective modules. As we have discussed,
Kaplansky~\cite{kaplansky-projective} proved that we can write $P_n = \oplus_{i
\in I_n} P_{n,i}$ for each $n$. Note that if $\kappa > \text{max}\{\,
|R| \, , \, \omega \,\}$ is a regular cardinal then $|P_{n,i}| <
\kappa$.

Next let $A$ be a given $R$-module. Using the natural isomorphism
\[
M \tensor_R (\oplus_{i \in \class{S}} N_i) \cong \oplus_{i \in
\class{S}} M \tensor_R N_i
\]
we may \emph{identify} $A \tensor_R P$ with the complex whose degree
$n$ is $\oplus_{i \in I_n} A \tensor_R P_{n,i}$. Moreover, for any
subcomplex $S \subseteq P$ of the form $S_n = \oplus_{i \in K_n}
P_{n,i}$ for some $K_n \subseteq I_n$ we can and will identify $A
\tensor_R S$ with the subcomplex of $A \tensor_R P$ whose degree $n$
is $\oplus_{i \in K_n} A \tensor_R P_{n,i} \subseteq \oplus_{i \in
I_n} A \tensor_R P_{n,i}$. We note that if $\kappa > \text{max}\{\,
|R| \, , \, \omega \,\}$ is a regular cardinal, then such a subcomplex
$S$ satisfies $|S| < \kappa$ whenever $|K_n| < \kappa$. Similarly, if
$\kappa > \text{max}\{\, |A| \, , \, \omega \,\}$ is a regular
cardinal, note that $|A \tensor_R S| < \kappa$ whenever $|K_n| <
\kappa$. We will use all of the above observations in the proof of our
theorem below.

\begin{theorem}\label{theorem-filtrations for complexes of projectives}
Let $A$ be a given $R$-module and take $\kappa > \text{max}\{\, |R| \,
, \, |A| \, , \, \omega \,\}$ to be a regular cardinal. Let $P$ be any
nonzero complex of projectives in which $A \tensor_R P$ is exact. Then
we can write $P$ as a continuous union $P = \cup_{\alpha < \lambda}
Q_{\alpha}$ where each $Q_{\alpha}, Q_{\alpha + 1}/Q_{\alpha}$ are
also $A \tensor_R -$ exact complexes of projectives and $|Q_{\alpha}|,
|Q_{\alpha + 1}/Q_{\alpha}| < \kappa$.
\end{theorem}

\begin{proof}
Write each $P_n = \oplus_{i \in I_n} P_{n,i}$ where $P_{n,i}$ are
countably generated. We prove the theorem in two steps.

\noindent (Step 1). We first show the following: We can find a nonzero
subcomplex $Q \subseteq P$ of the form $Q_n = \oplus_{i \in L_n}
P_{n,i}$ for some subcollections $L_n \subseteq I_n$ having $|L_n| <
\kappa$ and such that $A \tensor_R Q$ is exact.

Since $P$ is nonzero at least one $P_n \neq 0$. For this $n$, take any
nonempty $J_n \subseteq I_n$ having $|J_n| < \kappa$. Apply the
Covering Lemma with $P$ in the place of $X$ and taking the
subcollections to consist of this $J_n$ and all the other $J_n$ may be
empty. This gives us a nonzero subcomplex with $S^1_n = \oplus_{i \in
K^1_n} P_{n,i}$ for some subcollections $K^1_n \subseteq I_n$
satisfying $J_n \subseteq K^1_n$ and $|K^1_n| < \kappa$ for each $n$.

Now $A \tensor_R S^1$ is the subcomplex of $A \tensor_R P$ having $(A
\tensor_R S^1)_n = \oplus_{i \in K^1_n} A \tensor_R P_{n,i}$. That is,
the subcollections $K^1_n \subseteq I_n$ determine $A \tensor_R
S^1$. We now apply the Exact Covering Lemma with $A \tensor_R P$ in
the place of $Y$ and taking the subcollections to be the $K^1_n$. This
gives us an exact subcomplex $T^1 \subseteq A \tensor_R P$ with each
$T^1_n = \oplus_{i \in J^1_n} A \tensor_R P_{n,i}$ for some
subcollections $J^1_n \subseteq I_n$ satisfying $K^1_n \subseteq
J^1_n$ and $|J^1_n| < \kappa$.

But perhaps now the direct sums $\oplus_{i \in J^1_n} P_{n,i}$ don't
even form a \emph{subcomplex} of $P$ (because the tensor product with
$A$ may send some maps to $0$). So we again apply the Covering Lemma
to $P$ with the $J^1_n$ as the subcollections to find a subcomplex
$S^2 \subseteq P$ with each $S^2_n = \oplus_{i \in K^2_n} P_{n,i}$ for
some subcollections $K^2_n \subseteq I_n$ satisfying $J^1_n \subseteq
K^2_n$ and $|K^2_n| < \kappa$. Of course $S^1 \subseteq S^2$ because
$K^1_n \subseteq K^2_n$ for each $n$.

But now certainly $A \tensor_R S^2$ need not be exact, so we again
apply the Exact Covering Lemma to $A \tensor_R P$ taking the
subcollections to be the $K^2_n$. This gives us an exact subcomplex
$T^2 \subseteq A \tensor_R P$ with each $T^2_n = \oplus_{i \in J^2_n}
A \tensor_R P_{n,i}$ for some subcollections $J^2_n \subseteq I_n$
satisfying $K^2_n \subseteq J^2_n$ and $|J^2_n| < \kappa$. Notice that
we have $A \tensor_R S^1 \subseteq T^1 \subseteq A \tensor_R S^2
\subseteq T^2$ because $K^1_n \subseteq J^1_n \subseteq K^2_n
\subseteq J^2_n$.

And so it goes. The $\oplus_{i \in J^2_n} P_{n,i}$ need not form a
subcomplex of $P$. So we continue this back and forth, applying the
Covering Lemma to $P$ and the newly obtained subcollections $J^l_n$,
and then applying the Exact Covering Lemma to $A \tensor_R P$ and the
newly found subcollections $K^l_n$. We obtain an increasing sequence
of subcomplexes of $P$ $$0 \neq S^1 \subseteq S^2 \subseteq S^3
\subseteq \cdots$$ corresponding to the subcollections $J^1_n
\subseteq J^2_n \subseteq J^3_n \subseteq \cdots$. We also get an
increasing sequence of subcomplexes of $A \tensor_R P$
$$A \tensor_R S^1 \subseteq T^1 \subseteq A \tensor_R S^2 \subseteq
T^2 \subseteq A \tensor_R S^3 \subseteq  T^3 \subseteq \cdots$$
with each $T^l$ exact.

So we set $Q = \cup_{l \in \N} S^l$ and claim that $Q$ satisfies the
properties we sought. Indeed notice each $Q_n = \oplus_{i \in L_n}
P_{n,i}$ where $L_n = \cup_{l \in \N} J^l_n$. Also we still have
$|L_n| < \kappa$. Finally, since $A \tensor_R -$ commutes with direct
limits we get $A \tensor_R Q = \cup_{l \in \N} A \tensor_R S^l
=\cup_{l \in \N} T^l$. This complex is exact because each $T^l$ is
exact.

\

\noindent (Step 2). We now can easily finish to obtain the desired
continuous union. Start by finding a nonzero $Q^0 \subseteq P$ of the
form $Q^0_n = \oplus_{i \in L^0_n} P_{n,i}$ for some subcollections
$L^0_n \subseteq I_n$ having $|L^0_n| < \kappa$ and such that $A
\tensor_R Q^0$ is exact. Note that $Q^0$ and $P/Q^0$ are also
complexes of projectives and since $0 \xrightarrow{} Q^0
\xrightarrow{} P \xrightarrow{} P/Q^0 \xrightarrow{} 0$ is a
degreewise split short exact sequence, so is $0 \xrightarrow{} A
\tensor_R Q^0 \xrightarrow{} A \tensor_R P \xrightarrow{} A \tensor_R
P/Q^0 \xrightarrow{} 0$. It follows that $A \tensor_R P/Q^0$ must also
be exact. So if it happens that $P/Q^0$ is nonzero we can in turn find
a nonzero subcomplex $Q^1/Q^0 \subseteq P/Q^0$ with $Q^1/Q^0$ and
$(P/Q^0)/(Q^1/Q^0) \cong P/Q^1$ both $A \tensor_R -$ exact complexes
of projectives with cardinality less than $\kappa$. Note that we can
identify these quotients such as $P/Q^0$ as complexes whose degree $n$
entry is $\oplus_{i \in I_n-L_n} P_{n,i}$ and in doing so we may
continue to find an increasing union $0 \neq Q^0 \subseteq Q^1
\subseteq Q^2 \subseteq \cdots $ corresponding to a nested union of
subsets $L^0_n \subseteq L^1_n \subseteq L^2_n \subseteq \cdots$ for
each $n$. Assuming this process doesn't terminate we set $Q^{\omega} =
\cup_{\alpha < \omega} Q^{\alpha}$ and note that $Q^{\omega}_n =
\oplus_{i \in L^{\omega}_n} P_{n,i}$ where $L^{\omega}_n =
\cup_{\alpha < \omega} L^{\alpha}_n$. So still, $Q^{\omega}$ and
$P/Q^{\omega}$ are complexes of projectives and are $A \tensor_R -$
exact since $A \tensor_R -$ commutes with direct limits. Therefore we
can continue this process with $P/Q^{\omega}$ to obtain $Q^{\omega
+1}$ with all the properties we desire. Using this process we can
obtain an ordinal $\lambda$ and a continuous union $P = \cup_{\alpha <
\lambda} Q^{\alpha}$ with each $Q_{\alpha}, Q_{\alpha + 1}/Q_{\alpha}$
$A \tensor_R -$ exact complexes of projectives having $|Q_{\alpha}|,
|Q_{\alpha +1}/Q_{\alpha}| < \kappa$.
\end{proof}

We can now prove Theorem~\ref{thm-how to create projective on chain}

\begin{proof}
The plan is to apply Proposition~\ref{prop-how to create a projective
model structure}. First let $\kappa > \text{max}\{\, |R| \, , \, |A|
\, , \, \omega \,\}$ be a regular cardinal and let $S$ be the set of
all complexes $P \in \class{C}$ such that $|P| \leq \kappa$. (We
really need to take a representative for each isomorphism class so
that we actually get a set as opposed to a proper class). Since any
set $S$ cogenerates a complete cotorsion pair
$(\leftperp{(\rightperp{S})},\rightperp{S})$ it is enough to show
$\rightperp{S} = \rightperp{\class{C}}$. But this follows right away
from the chain complex version of Theorem~7.3.4
of~\cite{enochs-jenda-book}. The remaining remaining two properties of
Proposition~\ref{prop-how to create a projective model structure} to
check hold by straight duality of the proofs for the injective models
in Theorem~\ref{thm-how to create injective on chain}.
\end{proof}

\section{The Gorenstein AC-projective model structure on modules}\label{sec-Gor-proj}

We saw in Section~\ref{sec-Gorenstein-inj} that the exact AC-acyclic
model structure on $\ch$ gives rise to a Quillen equivalent model
structure on $\rmod $ in which the fibrant objects are the Gorenstein
AC-injectives.  One would then expect the exact firmly acyclic
model structure to give rise to a similar model structure on $\rmod$.
We construct this Gorenstein AC-projective model structure in this
section.

Recall that a module $M$ is \textbf{Gorenstein projective} if
$M=Z_{0}X$ for some totally acyclic complex of projectives; that is,
$X$ is exact and $\Hom (X,P)$ is exact for all projective modules $P$.
In view of Theorem~\ref{thm-Inj-exact-level}, we define $M$ to be
\textbf{Gorenstein AC-projective} if $M=Z_{0}X$ for some exact
firmly acyclic complex of projectives; that is, $X$ is exact and $\Hom
(X,F)$ is exact for all level (left) modules $F$, equivalently, $I \tensor_R X$ is exact for all AC-modules $I$.  Note that every Gorenstein
AC-projective is Gorenstein projective, and the two concepts agree if
every level module has finite projective dimension. On the other hand, over coherent rings 
the Gorenstein AC-projectives are exactly the Ding projectives from~\cite{gillespie-ding}. 

Our model structure on $\rmod$ will then have $\class{C}$ consist of
the Gorenstein AC-projective modules, $\class{F}$ consist of all
modules, and $\class{W}=\rightperp{\class{C}}$.  We now proceed as in
Section~\ref{sec-Gorenstein-inj}.

\begin{lemma}\label{lem-proj-cycles-of-W}
Let $R$ be a ring and suppose $Y$ is a complex of $R$-modules with
$H_{i}Y=0$ for $i>0$ and $Y_{i}$ level for $i<0$.  Then $Y$ is trivial
in the exact firmly acyclic model structure if and only if
$Y_{0}/B_{0}Y\in \class{W}$.
\end{lemma}

The proof of this lemma is very similar to the proof of
Lemma~\ref{lem-cycles-of-W}, suitably dualized.  We will set up the
outline, then leave the rest of the proof to the reader.

 \begin{proof}
Suppose $M$ is Gorenstein AC-projective, so that
$M=Z_{-1}X=B_{-1}X=X_{0}/B_{0}X$ for some exact firmly acyclic
complex of projectives $X$.  We claim that there is an isomorphism
\[
\Ext^{1} (X,Y) \xrightarrow{} \Ext^{1} (M,Y_{0}/B_{0}Y);
\]
this isomorphism would prove the lemma.  Indeed, because $X$ is a
complex of projectives, Lemma~\ref{lemma-homcomplex-basic-lemma} gives
us an isomorphism \[ \Ext^{1} (X,Y) \xrightarrow{} \ch (X, \Sigma
Y)/\sim, \] where $\sim$ denotes chain homotopy.  A chain map $\phi
\mathcolon X\xrightarrow{}\Sigma Y$ induces a map $B_{0}X\cong
X_{1}/B_{1}X\xrightarrow{}Y_{0}/B_{0}Y$.  A chain homotopy between
$\phi $ and $0$ gives us maps $D_{n}\mathcolon
X_{n}\xrightarrow{}Y_{n}$ with $-dD_{n}+D_{n-1}d=\phi _{n}$.  In
particular, $D_{0}d=\phi_{1}$ as a map from $X_{1}$ to $Y_{0}/B_{0}Y$.
Thus, there is a natural map
\begin{gather*}
 \ch (X, \Sigma Y)/\!\sim \xrightarrow{} \Hom (B_{0}X, Y_{0}/B_{0}Y)/\Hom
 (X_{0},Y_{0}/B_{0}Y)\\
\cong \Ext^{1} (M, Y_{0}/B_{0}Y).
 \end{gather*}

We now show this map is an isomorphism in analogous fashion to the
proof of Lemma~\ref{lem-cycles-of-W}.

 \end{proof}

Just as in the injective case, we then get the following scholium.

\begin{proposition}\label{prop-Gor-proj-induced}
For any ring $R$, suppose $M$ and $N$ are Gorenstein AC-projective
modules, with $M=Z_{0}X$ and $N=Z_{0}Y$ for $X$ and $Y$ exact
firmly acyclic complexes of projectives.  Given a map $f\mathcolon
M\xrightarrow{}N$, there is a chain map $\phi \mathcolon
X\xrightarrow{}Y$ with $Z_{0}\phi =f$.
\end{proposition}

This proposition then leads to the following lemma, analogous to
Lemma~\ref{lem-Gor-inj-retracts} with the dual proof.

\begin{lemma}\label{lem-Gor-proj-retracts}
For any ring $R$, the collection of Gorenstein AC-projective modules
is closed under retracts.
\end{lemma}

\begin{lemma}\label{lem-proj-spheres}
A module $M$ over a ring $R$ is in $\class{W}$ if and only if $S^{0}M$
is trivial in the exact firmly acyclic model structure.
\end{lemma}

\begin{proof}
A calculation using Lemma~\ref{lemma-homcomplex-basic-lemma} shows
that
\[
\Ext^{1} (X, S^{0}M) = \Ext^{1} (Z_{-1}X, M).
\]
for any exact complex of projectives $X$.
\end{proof}

\begin{theorem}\label{thm-Gor-proj-module}
For any ring $R$, there is an abelian model structure on $\rmod$, the
\textbf{Gorenstein AC-projective model structure}, in which every object
is fibrant and the cofibrant objects are the Gorenstein AC-projective
modules.
\end{theorem}

This model structure generalizes the Gorenstein projective model
structure for Gorenstein rings constructed in~\cite{hovey-cotorsion}, as well as
its generalization constructed in~\cite{gillespie-ding}.

The proof of Theorem~\ref{thm-Gor-proj-module} is dual to the proof of
Theorem~\ref{thm-Gor-module}, and so we leave it to the reader.

Just as in the injective case, we have the following lemma.

\begin{lemma}\label{lem-hereditary-proj}
For any ring $R$, the cotorsion pair $(\cat{C},\cat{W})$, where
$\cat{C}$ is the Gorenstein AC-projective modules, is hereditary, so
that $\Ext^{n} (C,W)=0$ for all $n>0$, $W\in \cat{W}$, and $C\in
\cat{C}$.  Hence the collection of Gorenstein AC-projective modules is
resolving \ulp closed under kernels of epimorphisms\urp .
\end{lemma}

\begin{proof}
Suppose $C$ is Gorenstein AC-projective, so that $F=Z_{0}X$, where $X$
is an exact firmly acyclic complex of projectives.  We have short
exact sequences
\[
0 \xrightarrow{} Z_{i}X \xrightarrow{} X_{i} \xrightarrow{}
Z_{i-1}X\xrightarrow{}0
\]
for all $i$, and each $Z_{i}X$ is also Gorenstein AC-projective.  A
simple computation then shows that
\[
\Ext^{n} (C, M) = \Ext^{1} (Z_{n-1}C, M),
\]
so $(\cat{C}, \cat{W})$ is hereditary.  It follows that $\cat{C}$ is
resolving.
\end{proof}

We then have the analogue of Theorem~\ref{thm-homotopy-Gorenstein},
whose proof is dual.

\begin{theorem}\label{thm-homotopy-Gorenstein-proj}
For any ring $R$, the homotopy category of the Gorenstein AC-projective
model structure is the quotient category of the category of Gorenstein
AC-projective modules obtained by identifying two maps when their
difference factors through a projective module.
\end{theorem}

Just as in the injective case, the Gorenstein AC-projective model
structure and the exact firmly acyclic model structure are
Quillen equivalent.

\begin{theorem}\label{thm-totally-Gorenstein-proj}
For any ring $R$, the functor $F\mathcolon \ch \xrightarrow{}\rmod $
defined by $F (X)=X_{0}/B_{0}X$ is a Quillen equivalence from the
exact firmly acyclic model structure to the the Gorenstein
AC-projective model structure.
\end{theorem}

\begin{proof}
The right adjoint of $F$ is $S^{0}$, which is exact and so obviously
preserves fibrations.  In fact, $S^{0}$ also preserves trivial
fibrations by Lemma~\ref{lem-proj-spheres}.  So $S^{0}$ is a right
Quillen functor.  To complete the proof, we will show that $F$
reflects weak equivalences between cofibrant objects and that the
natural map $FCS^{0}M\xrightarrow{}S^{0}M$ is a weak equivalence for
all modules $M$, where $C$ denotes cofibrant replacement in the exact
firmly acyclic model structure.  In view of Corollary~1.3.16
of~\cite{hovey-model}, this will complete the proof.

We will now show that $F$ reflects weak equivalences between cofibrant
objects.  By factoring any map between cofibrant objects into a
trivial cofibration followed by a fibration, we see that it suffices
to show that if $f\mathcolon X\xrightarrow{}Y$ is a fibration of
cofibrant objects such that $Ff$ is a weak equivalence, then $f$ is a
trivial fibration.  So we are given a short exact sequence
\[
0 \xrightarrow{} K \xrightarrow{} X \xrightarrow{f} Y \xrightarrow{} 0
\]
with $X$ and $Y$ exact firmly acyclic complexes of projectives.  This
sequence is necessarily degreewise split, so $K$ is also an exact
firmly acyclic complex of projectives.  The functor $F$ is right exact but
not exact, but since $Y$ is exact we do get a short exact sequence
\[
0 \xrightarrow{} K_{0}/B_{0}K \xrightarrow{} X_{0}/B_{0}X
\xrightarrow{} Y_{0}/B_{0}Y \xrightarrow{} 0.
\]
Since $Ff$ is a weak equivalence, $K_{0}/B_{0}K$ is in $\cat{W}$.
Lemma~\ref{lem-proj-cycles-of-W} then implies that $K$ is trivial in
the exact firmly acyclic model structure, so $f$ is a trivial
fibration.

Now take any module $M$, and let $CS^{0}M$ be a cofibrant replacement
for $S^{0}M$, so that we have a short exact sequence
\[
0\xrightarrow{} Y \xrightarrow{} CS^{0}M \xrightarrow{} S^{0}M
\xrightarrow{} 0
\]
with $Y$ trivial in the exact firmly acyclic model category.
Because $M$ is concentrated in degree $0$, one can check that we get
an exact sequence
\[
0 \xrightarrow{} FY \xrightarrow{} FCS^{0}M \xrightarrow{} M
\xrightarrow{} 0
\]
Furthermore, $Y_{i}$ is projective for all $i\neq 0$ and $H_{i}Y=0$ for
all $i\neq 0$, so Lemma~\ref{lem-cycles-of-W} implies that $FY\in
\class{W}$.  Hence $FCS^{0}M\xrightarrow{}M$ is a weak
equivalence.
\end{proof}

The functor
\[
\gamma : \rmod \xrightarrow{}\Ho \rmod
\]
to the homotopy category of the Gorenstein AC-projective model
structure is an exact functor to a triangulated category that
preserves products and sends all level modules and all injective
modules to $0$.

Just as in the injective case, it also is initial in the following
sense.

\begin{proposition}\label{prop-initial-proj}
The homotopy category of the Gorenstein AC-projective model structure
is initial among all triangulated categories with an exact functor
from $\rmod$ that preserves products and sends all elements of
$\cat{W}$ to zero.
\end{proposition}

The proof is exactly the same as in the injective case.

Unfortunately, as in the injective case, we do not what $\cat{W}$ is.
Ideally, it would be the smallest thick subcategory closed under
products that contains the level and injective modules, but we do not
know if this is true.

We do know that $(\cat{C},\cat{W})$ is cogenerated by a set though.

\begin{proposition}\label{prop-cogenerated-proj}
For any ring $R$, the cotorsion pair $(\cat{C},\cat{W})$, where
$\cat{C}$ is the class of Gorenstein AC-projectives, is cogenerated by
a set.  Thus the Gorenstein AC-projective model structure is
cofibrantly generated.
\end{proposition}

The proof of this proposition is very similar to the proof of
Proposition~\ref{prop-cogenerated}, so we leave it to the reader.

There is a relationship between the Gorenstein AC-projective and the
Gorenstein AC-injective model structures.

\begin{proposition}\label{prop-relation}
For any ring $R$, the identity functor is a left Quillen functor from
the Gorenstein AC-projective model structure to the Gorenstein
AC-injective model structure.  It is a Quillen equivalence when $R$ is
Gorenstein.
\end{proposition}

\begin{proof}
It is clear that the identity functor takes cofibrations in the
Gorenstein AC-projective model structure, which are certain
monomorphisms, to cofibrations in the Gorenstein AC-injective model
structure, which are all mononmorphisms.  Similarly, it takes
fibrations in the Gorenstein AC-injective model structure to
fibrations in the Gorenstein AC-projective model structure.  Together,
these make it a left Quillen functor as required.

When $R$ is Gorenstein, these model structure coincide with the ones
constructed in~\cite{hovey-cotorsion}, where it is shows that they are Quillen
equivalent.
%% Paragraph above has typos
\end{proof}

\appendix
\section{Complexes of projectives}\label{sec-appendix}

The object of this section is to prove
Theorem~\ref{thm-Inj-exact-level}, which we will restate below as
Corollary~\ref{cor-Inj-exact-level}, and some related theorems.  To
prove these results, we will use the following theorem of independent
interest.

\begin{theorem}\label{thm-pure}
Let $R$ be a ring and let $C$ be a complex of projective $R$-modules.
If $Y$ is a pure exact complex of $R$-modules, then $\homcomplex (C,Y)$
is exact, or, equivalently, $\Ext^{1}_{\ch} (C,Y)=0$.  Similarly, if $Z$ is
a pure exact complex of right $R$-modules, then $Z\otimes_{R}C$ is
exact.  
\end{theorem}

This theorem generalizes Neeman's result~\cite{neeman-flat} that $\Ext^{1}
(C,Y)=0$ if $C$ is a complex of projectives and $Y$ is a complex of
flat modules with flat cycles, as such $Y$ are automatically pure.
However, Neeman also proves that such $Y$ are the only complexes of
flat modules with $\Ext^{1} (C,Y)=0$ for all complexes of projectives
$X$.  The generalization of this fact is untrue; it is easy to see by
induction that every chain map from a complex of projectives to a
bounded below exact complex is chain homotopic to $0$, and not all
bounded below exact complexes are pure exact.

Our first goal is to reduce the study of all complexes of projectives
to a manageable set of them.

\begin{lemma}\label{lemma-complexes of projectives are retracts of complexes of frees}
For any ring $R$, let $P$ be a complex of projective $R$-modules. Then
$P$ is a retract of a complex $F$ of free modules. Furthermore, if $P$
is exact then $F$ can be taken to be exact.
\end{lemma}

\begin{proof}
Recall Eilenberg's swindle (Corollary~2.7 of~\cite{lam}) allows one to
construct, for any projective module $P$ a free module $F$ such that
$P \oplus F \cong F$. Given a complex of projectives $P$, we use the
swindle to find for each $P_n$ a free $F_n$ such that $P_n \oplus F_n
\cong F_n$. Then $P \oplus (\oplus_{n \in \Z} D^n(F_n))$ is a complex
of free modules. Indeed in degree $n$ the complex equals $P_n \oplus
F_n \oplus F_{n+1} \cong F_n \oplus F_{n+1}$ which is free. Of course
$P$ is a retract of $P \oplus (\oplus_{n \in \Z} D^n(F_n))$ by
construction and also $P \oplus (\oplus_{n \in \Z} D^n(F_n))$ is exact
whenever $P$ is exact.
\end{proof}

\begin{theorem}\label{thm-bounded above complexes of finitely generated frees cogenerate}
For any ring $R$, the cotorsion pair $(\class{C},\class{W})$, where
$\class{C}$ is the class of all complexes of projective modules, is
cogenerated by the collection of all bounded above complexes of
finitely generated free modules.
\end{theorem}

\begin{proof}
 Let $\class{S}$ be the set of bounded above complexes of finitely
generated free modules. It cogenerates a cotorsion pair
$(\leftperp{(\rightperp{\class{S}})}, \rightperp{\class{S}})$ and we
wish to show $\class{C} = \leftperp{(\rightperp{\class{S}})}$. Since
$\class{S}$ contains a set of projective generators we know that
$\leftperp{(\rightperp{\class{S}})}$ is precisely the class of all
retracts of transfinite extensions of objects in
$\class{S}$. (Although complexes of projective modules are not closed
under all direct limits, they are closed under transfinite
compositions, just as projective modules are).  Since
Lemma~\ref{lemma-complexes of projectives are retracts of complexes of
frees} tells us that any complex of projectives is a retract of a
complex of free modules we only need to show that a complex of free
modules is a transfinite composition of bounded above complexes of
finitely generated free modules.  So let $F$ be a complex of free
modules and write each $F_n = \oplus_{i \in I_n} R_i$ for some $I_n$
and each $R_i = R$. We do a simplified version of the argument in
Lemma~\ref{lemma-covering lemma}. Assuming $F$ is nonzero we can find
a nonzero $F_n$ and we take just one summand $R_j$ for some $j \in
I_n$. We start to build a bounded above subcomplex $X \subseteq F$ by
setting $X_n = R_j$ and setting $X_i = 0$ for all $i > n$. Now note
$d(R_j) \subseteq \oplus_{i \in I_{n-1}} R_i$ and set $L_{n-1} = \{ \,
i \in I_{n-1} \, | \, d(R_j) \cap \oplus_{i \in I_{n-1}} R_i \neq 0 \,
\}$. We set $X_{n-1} = \oplus_{i \in L_{n-1}} R_i$ and note that
$|L_{n-1}|$ must be finite.  We can continue down in the same way
finding $L_{n-2} \subseteq I_{n-2}$ with $|L_{n-2}|$ finite and with
$d(\oplus_{i \in L_{n-1}} R_i) \subseteq \oplus_{i \in L_{n-2}}
R_i$. In this way we get a subcomplex of $X$: $$X^n = \cdots
\xrightarrow{} 0 \xrightarrow{} R_j \xrightarrow{} \oplus_{i \in
L_{n-1}} R_i \xrightarrow{} \oplus_{i \in L_{n-2}} R_i \xrightarrow{}
\cdots $$ So $X$ is a nonzero bounded above complex of finitely
generated free modules.

Now following the method of Step 2 in the proof of
Theorem~\ref{theorem-filtrations for complexes of projectives} see
that we can write any complex of free modules as a continuous union of
bounded above complexes of finitely generated free modules.
\end{proof}

We would now like to find complexes in
$\class{W}=\rightperp{\class{C}}$.  Since $\class{C}$ consists of
complexes of projectives, this is equivalent to finding complexes $Y$
such that every chain map from a complex of projectives $C$ to $Y$ is
chain homotopic to $0$.  In view of the above lemma, we can assume
that $C$ is a bounded above complex of finitely generated free
modules.  Of course, $Y$ must be exact, since $\class{C}$ contains
$S^{n} (R)$ for all $n$.  One might guess that $Y$ should be slightly
better than exact; the content of Theorem~\ref{thm-pure} says that any
pure exact $Y$ is in $\class{W}$.

\begin{lemma}\label{lem-pure}
Given a ring $R$, suppose $X$ is a bounded complex of finitely
presented modules and $Y$ is a pure exact complex.  Then every chain
map $f\mathcolon X\xrightarrow{}Y$ is chain homotopic to $0$.
Similarly, if $Z$ is a pure exact complex of right $R$-modules, then
$Z\otimes_{R}X$ is exact.  
\end{lemma}

The second statement is actually true for any bounded complex $X$,
even if the entries are not finitely presented.  

\begin{proof}
We begin with the first statement.  Let $n$ be the number of $i$ for
which $X_{i}$ is nonzero.  If $n=1$, the result follows by definition
of pure exactness.  In general, let $m$ be the largest degree $i$ for which
$X_{i}$ is nonzero, and let $A$ be the subcomplex of $X$ with
$A_{i}=X_{i}$ for $i<m$ and $A_{m}=0$.  The short exact sequence
\[
0 \xrightarrow{} A \xrightarrow{} X \xrightarrow{}  S^{m}X_{m}
\xrightarrow{} 0
\]
is degreewise split, so the induced sequence
\[
0 \xrightarrow{}\homcomplex (S^{m}X_{m},Y) \xrightarrow{} \homcomplex (X,Y)
\xrightarrow{} \homcomplex (A,Y) \xrightarrow{} 0
\]
is still exact.  The long exact sequence in homology now gives us the
result by induction.

The proof of the second statement is virtually identical.  The base
case of $n=1$ again follows by definition of pure exactness, and the
proof of the induction step is the same except for replacing $\homcomplex
(-,Y)$ by $Z\otimes_{R}-$.  
\end{proof}

We can now prove Theorem~\ref{thm-pure}.  

\begin{proof}[Proof of Theorem~\ref{thm-pure}]
We first show that $\homcomplex (C,Y)$ is exact for $C$ a complex of
projectives and $Y$ pure exact.  In view of Theorem~\ref{thm-bounded
above complexes of finitely generated frees cogenerate}, we may assume
that $C$ is a bounded above complex of finitely generated free
modules.  It suffices to show that any chain map $f\mathcolon
C\xrightarrow{}Y$ is chain homotopic to $0$.  We construct a chain
homotopy $D_{n}\mathcolon C_{n}\xrightarrow{}Y_{n+1}$ with
$dD_{n}+D_{n-1}d=f_{n}$ by downwards induction on $n$.  Since $C$ is
bounded above, we can take $D_{n}=0$ for large $n$ to begin the
induction.  So we suppose that $D_{i}$ has been defined for $i\geq n$
and that $dD_{n+1}+D_{n}d=f_{n+1}$.

We will first modify $D_{n}$ to a new map $\widetilde{D}_{n}$ so that
this identity still holds, and then construct $D_{n-1}$ such that
$d\widetilde{D}_{n}+D_{n-1}d=f_{n}$. Note first that
\[
(f_{n}-dD_{n})d = d (f_{n+1}-D_{n}d)=d^{2}D_{n+1}=0,
\]
so there is an induced map $g_{n}\mathcolon
C_{n}/B_{n}C\xrightarrow{}Y_{n}$.  Now consider the bounded complex
$X$ of finitely presented modules with $X_{n}=C_{n}/B_{n}C$,
$X_{n-1}=C_{n-1}$, $X_{n-2}=C_{n-1}/B_{n-1}C$, and $X_{i}=0$ for all
other $i$.  There is a chain map $g\mathcolon X\xrightarrow{}Y$ that
is $g_{n}$ in degree $n$, $f_{n-1}$ in degree $n-1$, and $f_{n-2}d$ in
degree $n-2$.  By the preceding lemma, this chain map must be chain
homotopic to $0$.  This gives us maps $D_{n}'\mathcolon
C_{n}/B_{n}C\xrightarrow{}Y_{n+1}$ and $D_{n-1}\mathcolon
C_{n-1}\xrightarrow{}Y_{n}$ such that
$dD_{n}'+D_{n-1}d=f_{n}-dD_{n}$.  Put another way, this means that
\[
d (D_{n}+D_{n}')+D_{n-1}d =f_{n},
\]
as required.  Furthermore, we still have the required relation
\[
dD_{n+1} + (D_{n}+D_{n}')d=f_{n+1}
\]
because $D_{n}'d=0$.

For the second half of the theorem, we need to show that
$Z\otimes_{R}C$ is exact, where $Z$ is any pure exact complex of right
$R$-modules.  Since homology commutes with direct limits, we can again
assume that $C$ is a bounded above complex of finitely generated free
modules.  But any bounded above complex is a direct limit of its
truncations $C^{-n}$, where $(C^{-n})_{k}=C_{k}$ if $k> -n$,
$(C^{-n})_{-n}=B_{-n}C$, and $(C^{-n})_{k}$ is $0$ otherwise.  Each
truncation $C^{-n}$ is a bounded complex, so $Z\otimes_{R}C^{-n}$ is
exact by Lemma~\ref{lem-pure}.  Therefore the direct limit
$Z\otimes_{R}C$ is exact.  
\end{proof}

Noew recall the character module of of an $R$-module $M$ is $\Hom_{\Z}
(M, \Q )$.  

\begin{proposition}\label{prop-dual-exact}
Let $R$ be a ring, $C$ a chain complex, and $M$ a right $R$-module.
Then $M\otimes_{R}C$ is exact if and only if $\Hom_{R}(C, M^{+})$ is
exact.  
\end{proposition}

\begin{proof}
Note that $M\otimes_{R}C$ is exact if and only if
$(M\otimes_{R}C)^{+}$ is exact, and 
\[
(M\otimes_{R}C)^{+}\cong \Hom_{R} (C, M^{+}).  
\]
\end{proof}

If $R$ is a ring, $\cat{C}$ is a collection of right $R$-modules, and
$\cat{D}$ is a collection of left $R$-modules, we say that
$(\cat{C},\cat{D})$ is a \textbf{duality pair} if $M\in \cat{C}$ if
and only if $\cat{M}^{+}$ is in $\cat{D}$, and $N\in \cat{D}$ if and
only if $N^{+}\in \cat{C}$.  

\begin{theorem}\label{thm-dual-exact}
Let $R$ be a ring, and suppose $(\cat{C},\cat{D})$ is a duality pair
such that $\cat{D}$ is closed under pure quotients.  Let $C$ be a
complex of projectives.  Then $M\otimes_{R}C$ is exact for all $M\in
\cat{C}$ if and only if $\Hom_{R} (C,N)$ is exact for all $n\in
\cat{D}$.  
\end{theorem}

\begin{proof}
In view of Proposition~\ref{prop-dual-exact}, if $\Hom_{R} (C,N)$ is
exact for all $N\in \cat{D}$, then $M\otimes_{R}C$ is exact for all
$M\in \cat{C}$.  Conversely, suppose $M\otimes_{R}C$ is exact for all
$m\in \cat{C}$.  Then if $N\in \cat{D}$,$N^{+}\otimes_{R}C$ is exact,
and so Proposition~\ref{prop-dual-exact} tells us that $\Hom (C,
N^{++})$ is exact.  We conclude that $\Hom_{R} (C, K)$ is exact for
all $K\in \cat{D}^{++}$, and we note that $\cat{D}^{++}\subseteq
\cat{D}$ since $(\cat{C},\cat{D})$ is a duality pair.  

Now, for any $N$, the natural map $N\xrightarrow{}N^{++}$ is a pure
monomorphism~\cite[Proposition~5.3.9]{enochs-jenda-book}.  So if $N\in
\cat{D}$, the quotient $N^{++}/N$ is also in $\cat{D}$ since $\cat{D}$
is closed under pure quotients.  We can therefore create a resolution
of $N\in \cat{D}$ by elements of $\cat{D}^{++}$.  That is, we can
find a pure exact chain complex $X$ where $X_{i}=0$ for $i>0$,
$X_{0}=N$, and each of the $X_{i}$ for $i<0$ is in $\cat{D}^{++}$.
This gives a short exact sequence 
\[
0 \xrightarrow{} S^{0}N \xrightarrow{} X \xrightarrow{} Y
\xrightarrow{} 0
\]
in which $X$ is pure exact and $Y$ is a bounded above complex with
entries in $\cat{D}^{++}$.  Theorem~\ref{thm-pure} tells us that
$\homcomplex (C,X)$ is exact.  So to complete the proof it will suffice
to show that $\homcomplex (C,Y)$ is exact.  If $Z$ is a \emph{bounded}
complex with entries in $\cat{D}^{++}$,
then we can prove $\homcomplex (C,Z)$ is exact by induction on the number
of nonzero entries in $Z$.  In general, any bounded above complex $Y$
is the inverse limit of its truncations $Y^{-n}$ for $n\in \Z$, where
$(Y^{-n})_{i}=Y_{i}$ for $i\geq -n$ and is $0$ otherwise.  This is a
very simple inverse limit, and so it is easy to check that $\homcomplex
(C,Y)=\invlim \homcomplex (C, Y^{-n})$.  It follows that $\homcomplex (C,Y)$
is exact, completing the proof.  
\end{proof}

We now recover Theorem~\ref{thm-Inj-exact-level}.

\begin{corollary}\label{cor-Inj-exact-level}
For any ring $R$, a complex of projectives $C$ is AC-acyclic if and
only if it is firmly acyclic.  If level $R$-modules all have
finite projective dimension, these conditions are equivalent to
$\Hom_{R} (C,P)$ being exact for all projective $R$-modules $P$.
\end{corollary}

\begin{proof}
The classes of absolutely clean modules and level modules form a
duality pair by Theorem~\ref{thm-duality}.  For the second statement,
note that the collection of all modules $M$ such that $\Hom_{R} (C,M)$
is exact is a thick subcategory, because $C$ is a complex of
projectives.  So if it contains projective modules, it contains all
modules of finite projective dimension.  
\end{proof}

%  \bibliography{hovey}
%  \bibliographystyle{amsalpha}

\providecommand{\bysame}{\leavevmode\hbox to3em{\hrulefill}\thinspace}
\providecommand{\MR}{\relax\ifhmode\unskip\space\fi MR }
\providecommand{\MRhref}[2]{%
  \href{http://www.ams.org/mathscinet-getitem?mr=#1}{#2}
}
\providecommand{\href}[2]{#2}

% \affiliationone{
%  Daniel Bravo \\
%  Universidad Austral de Chile \\
%  Valdivia, Chile 
%  \email{danielbravovivallo@gmail.com}}
% \affiliationtwo{
% James Gillespie \\
% Ramapo College of New Jersey \\
% School of Theoretical and Applied Science \\
% 505 Ramapo Valley Road \\
% Mahwah, NJ 07430 USA 
% \email{jgillesp@ramapo.edu}}
% \affiliationthree{
% Mark Hovey \\
% Wesleyan University \\
% Middletown, CT 06459 USA 
% \email{mhovey@wesleyan.edu}}

\end{document}